\documentclass[11pt]{amsart}
\usepackage{amsfonts}
\usepackage{amsthm}
\usepackage{geometry}
\usepackage{amsfonts}
\usepackage{amsmath}
\usepackage{xcolor}

\newcommand{\EEQ}{\end{equation}}
\newcommand{\rfb}[1]{\mbox{\rm
   (\ref{#1})}\ifx\undefined\stillediting\else:\fbox{$#1$}\fi}

                         \newcommand{\ud}     {{\rm d}}
\newcommand{\bt}{\begin{Theorem}}
\newcommand{\et}{\end{Theorem}}
\newcommand{\br}{\begin{remark}}
\newcommand{\er}{\end{remark}}
\newcommand{\bc}{\begin{Corollary}}
\newcommand{\ec}{\end{Corollary}}
\newcommand{\el}{\end{Lemma}}
\newcommand{\bd}{\begin{definition}}
\newcommand{\ed}{\end{definition}}

\newcommand{\pd}[1]{\langle #1\rangle}

\newcommand{\N}  {\mathbb{N}}
\newcommand{\R}  {\mathbb{R}}

\newcommand{\mm}    {{\hbox{\hskip 0.5pt}}}

\newcommand{\bluff} {{\hbox{\raise 15pt \hbox{\mm}}}}

\newfont{\Blackboard}{msbm10 scaled 1200}

\newfont{\roma}{cmr10 scaled 1200}

\def\CC{\rm \hbox{C\kern-.56em\raise.4ex
         \hbox{$\scriptscriptstyle |$}\kern+0.5 em }}
\newtheorem{corollary}{Corollary}[section]
\newtheorem{definition}[corollary]{Definition}

\newtheorem{remark}[corollary]{Remark}
\numberwithin{equation}{section}
\newtheorem{lem}{Lemma}[section]
\newtheorem{prop}{Proposition}[section]

\newtheorem{thm}{Theorem}[section]
\newtheorem{rem}{Remark}[section]

\def\ds{\displaystyle}
\newcommand{\re}{\mathrm{Re}}

\begin{document}
\thispagestyle{empty}
  
\title[Stability of a degenerate thermoelastic equation]{Stability of a degenerate thermoelastic equation}

\author{Ka\"{\i}s AMMARI}
\address{LR Analysis and Control of PDEs, LR 22ES03, Department of Mathematics, Faculty of Sciences of Monastir, University of Monastir, 5019 Monastir, Tunisia and LMV/UVSQ/Paris-Saclay, France} \email{kais.ammari@fsm.rnu.tn}

\author{Fathi Hassine}
\address{LR Analysis and Control of PDEs, LR 22ES03, Department of Mathematics, Faculty of Sciences of Monastir, University of Monastir, Tunisia}\email{fathi.hassine@fsm.rnu.tn}

\author{Luc ROBBIANO}
\address{Laboratoire de Math\'ematiques, Universit\'e de Versailles Saint-Quentin en Yvelines, 78035 Versailles, France}
\email{luc.robbiano@uvsq.fr}

\begin{abstract}
This work is dedicated to the study of a linear model arising in thermoelastic rod of homogeneous material. The system is resulting from a coupling of a heat and a wave equation in the interval $(0,1)$ with Dirichlet boundary conditions at the outer endpoints where the parabolic component is degenerating at the end point $x=0$. Two models are considered the first is with weak degeneracy and the second is with strong degeneracy. We aim to study the well-posedness and asymptotic stability of both systems using techniques from the $C_{0}$-semigroup theory and a use a
frequency domain approach based on the well-known result of Pr\"uss in order to prove using some multiplier techniques that the energy of classical solutions decays uniformly as time goes to infinity.
\end{abstract}    


\maketitle

\tableofcontents
 
\section{Introduction}
The equations of thermoelasticity describe the elastic and the thermal behavior of elastic, heat conductive media, in particular the reciprocal actions between elastic stresses and temperature differences. The equations are a coupling of the equations of elasticity and of the heat equation and thus build a hyperbolic-parabolic system. Indeed, both hyperbolic and parabolic effects are encountered. Thermoelasticity is the elasticity of bodies resulting from an increase in temperature. Thermoelasticity is the appropriate model for the explanation of the decay of the amplitude of free vibrations of some elastic bodies. Since then a wide variety of results and applications have been obtained in many different fields. For instance, in \cite{CVC1,CVC2,C,SHTBSA} synthetic tissues which mimic human bones are investigated, while in \cite{RFG} cardiological tissues are considered. Besides a model of apples regarded as thermoelastic bodies is studied in \cite{MKPK,SA}. Also, even the relevance of studying viscoelasticity problems in connection with earthquakes is highlighted in \cite{Bor}.

\medskip

Several mathematical models that come from physics lead to the study of partial differential equations and sometimes evolution equations allowing mathematicians to describe the behavior of a quantity that depends on several variables (states or entries). These kind of equations are called coupled systems. The problem concerns the stabilization of a coupled system has been intensively studied the last three decays. In the case of coupled wave-wave systems subjected to only one internal feedback operator (this is the indirect internal stabilization situation), positive and negative exponential and polynomial stability results have been obtained in \cite{Ala-Can-Com,AHR,AH,ALS,hassine1,Kap} for local and non singular damping. Some of these results were extended to several cases (plate-plate, wave-plate, coupling) for locally distributed indirect stabilization in \cite{hassine2,hassine4,hassine3,Bey}. A large number of papers, see for instance \cite{Amm-Bad-Ben,Daf,Las,Leb-Zua,Liu-Liu,Muniz-Racke,Zha-Zua} just to cite a few, concern the stabilization of hyperbolic-parabolic coupled systems, such as thermoelasticity, thermoplates, fluid-structure interaction etc. For such systems, the main goal is to determine whether the dissipation induced by the heat-type equation is sufficient for stabilizing the full system obtained by coupling it to a hyperbolic-type equation. 

A deformation of a thermoelastic body which varies in time leads to a change of the temperature distribution in the body, and conversely. The internal energy of the body depends on both the deformation and the temperature. A thermoelastic system describes the above coupled processes. In general, it consists of an elastic equation and a heat equation which are coupled in a fashion such that the transfer between the mechanical energy and the heat energy is taken into account \cite{Now}. In this paper we consider the linear system of thermoelasticity with Dirichlet boundary conditions which describe a linear thermoelastic rod of homogeneous material. If $u(x,t)$ and $\theta(x,t)$ for $x\in(0,1)$ and $t\geq0$  denote the unknown functions representing the transverse displacement of the string and the temperature difference to a fixed reference temperature, respectively, the differential equations for $u$ and $\theta$ are
\medskip
\begin{equation}\label{ITE1}
\left\{
\begin{array}{ll}
\ddot{u}(x,t)-u''(x,t)+\kappa\theta'(x,t)=0&(x,t)\in(0,1)\times(0,+\infty),
\\
\dot{\theta}(x,t)-\left(a(x)\theta'(x,t)\right)'+\kappa\dot{u}'(x,t)=0&(x,t)\in(0,1)\times(0,+\infty),
\\
u(0,t)=\theta(0,t)=0&t\in(0,+\infty),
\\
u(1,t)=\theta(1,t)=0&t\in(0,+\infty),
\\
\ds u(x,0)=u^{0}(x),\;\dot{u}(x,0)=u^{1}(x),\;\theta(x,0)=\theta^{0}(x)&x\in(0,1),
\end{array}
\right.
\end{equation}
where the point stands for the time derivative and the prime stands for the space derivative, $\kappa$ is a positive constant small enough and $a$ is the coefficient damping which assumed to be a non-negative function. Which can be regarded as the thermal effect from the heat equation $\dot{\theta}-(a\theta')'=0$ acting on the classical wave equation $\ddot{u}-u'' = 0$ by the lower order coupling term $(\kappa\theta',\kappa\dot{u}')$.

\medskip

The natural energy of system \eqref{ITE1} is given by
$$
E(t)=\frac{1}{2}\int_{0}^{1}\left(|\dot{u}(x,t)|^{2}+|u'(x,t)|^{2}+|\theta(x,t)|^{2}\right)\,\ud x
$$
and it is dissipated according to the following law
\begin{equation*}
\frac{\ud}{\ud t}E(t)=-\int_{0}^{1}a(x)\,|\theta'(x,t)|^{2}\,\ud x,\;\forall\,t>0.
\end{equation*}

In the last two decades, some qualitative properties (containing well-posedness, blow-up, propagation of singularity, decay property and asymptotic behavior) of solutions to some classical thermoelastic systems have caught a lot of attention. In \cite{Amm-Bad-Ben} the authors give sufficient conditions which allow the study of the exponential stability of an abstract systems closely related to the linear thermoelasticity systems by a decoupling and multiplier techniques. In \cite{Amm-Liu-Shel,Chain-Trig,Fer-Liu-Racke,Hao-Liu,Liu-Liu,Hao-Liu-Yong} an abstract framework for analysis of linear thermoelastic systems was considered where sufficient conditions for the exponential stability, analyticity and differentiability of the associated semigroups for the thermoelastic systems are obtained. Lebeau and Zuazua \cite{Leb-Zua} consider the two and three-dimensional system of linear thermoelasticity in a bounded smooth domain with Dirichlet boundary conditions. They analyze whether the energy of solutions decays exponentially uniformly to zero in particular they show that the decay rate is never uniform when the domain is convex and in two space dimensions they give a sufficient condition for the uniform decay to hold. Recently, Chen and Liu \cite{Chen-Liu} focus  on the generalized thermoelastic system of heat equation and wave equation by the lower order coupling term. They derive new large time asymptotic profiles with the regularity-loss structure and optimal decay estimates with suitable higher regularities for the Cauchy data. 

\medskip

All the previous works deal with the case where the parabolic component is non-degenerate. Stability properties of non-degenerate equations have been widely studied in recent years.  However, many problems that are relevant for applications are described by degenerate equations. For example, degenerate parabolic equations arise from many problems with physical background. For instance, the velocity field of a laminar flow on a flat plate can be described by the Prandtl equations \cite{Car-Mar-Van1,Car-Mar-Van2}. By using the so-called Crocco transformation, these equations are transformed into a nonlinear degenerate parabolic equation (the Crocco equation \cite{Ole-Sam}). In this paper we are interested in the degenerated case where the degeneracy occurs at the boundary. Mainly, we consider two kinds of degeneracy equations: a weak degeneracy (WDP) and a strong degeneracy (SDP) and we aim to study the well-posedness and stability of both systems. In particular, let's assume that
\begin{equation}\label{ITE2}
a\in\mathcal{C}^{0}([0,1])\cap\mathcal{C}^{1}((0,1]),\quad a>0 \text{ in } (0,1] \text{ and } a(0)=0,
\end{equation}
if in addition $a$ satisfies
\begin{equation}\label{ITE3}
\frac{1}{a}\in L^{1}(0,1),
\end{equation}
we consider the weakly degenerate problem (WDP):
\begin{equation}\label{ITE4}
\left\{
\begin{array}{ll}
\ddot{u}(x,t)-u''(x,t)+\kappa\theta'(x,t)=0&(x,t)\in(0,1)\times(0,+\infty),
\\
\dot{\theta}(x,t)-\left(a(x)\theta'(x,t)\right)'+\kappa\dot{u}'(x,t)=0&(x,t)\in(0,1)\times(0,+\infty),
\\
u(0,t)=\theta(0,t)=0&t\in(0,+\infty),
\\
u(1,t)=\theta(1,t)=0&t\in(0,+\infty),
\\
\ds u(x,0)=u^{0}(x),\;\dot{u}(x,0)=u^{1}(x),\;\theta(x,0)=\theta^{0}(x)&x\in(0,1),
\end{array}
\right.
\end{equation}
and if in addition to \eqref{ITE2}, $a$ satisfies
\begin{equation}\label{ITE5}
a\in\mathcal{C}^{1}([0,1]),\quad \frac{1}{a}\notin L^{1}(0,1) \text{ and }\frac{1}{\sqrt{a}}\in L^{1}(0,1),
\end{equation}
we consider the strong degenerate problem (SDP):
\begin{equation}\label{ITE6}
\left\{
\begin{array}{ll}
\ddot{u}(x,t)-u''(x,t)+\kappa\theta'(x,t)=0&(x,t)\in(0,1)\times(0,+\infty),
\\
\dot{\theta}(x,t)-\left(a(x)\theta'(x,t)\right)'+\kappa\dot{u}'(x,t)=0&(x,t)\in(0,1)\times(0,+\infty),
\\
u(0,t)=(a\theta')(0,t)=0&t\in(0,+\infty),
\\
u(1,t)=\theta(1,t)=0&t\in(0,+\infty),
\\
\ds u(x,0)=u^{0}(x),\;\dot{u}(x,0)=u^{1}(x),\;\theta(x,0)=\theta^{0}(x)&x\in(0,1).
\end{array}
\right.
\end{equation}

The controllability for degenerate parabolic equations has been extensively studied in recent years \cite{Gueye,Car-Mar-Van1,Car-Mar-Van2,Car-Fra-Roc,Car-Mar-Van3,Wang1,Wang2} and references therein (see also \cite{{Ala-Can-Leu}} for the hyperbolic equation). However, few attentions were paid to the corresponding stability problem \cite{Kin-Par,Wang1}. Wang \cite{Wang1} investigated the boundary behavior at $x=0$ and asymptotic behavior \cite{Gao-Li-Liu,Wang1} as $t$  goes to the infinity. In both cases of weak and strong degeneracy, uniform exponential stability was obtained. These works extended the result of the non-degeneracy case as it is well known that the heat equation without degeneracy is exponential stable. Things are quite different when it comes to coupled systems involving parabolic equation. In fact, a lack of uniform stability occurs in several situations. For instance, for the fluid-structure interaction model a sharp estimates for the rate of energy decay of a certain class of classical solutions was provided in one-dimension (see \cite{Bat-Pau-Sei}) and multi-dimension case (see \cite{Ava-Las-Tri,Ng-Sei,Duc}). Regarding the degenerate case for this same model it was shown that when the degeneracy is weak, the underlying semigroup is polynomially stable, and a sharp polynomial decay rates was provided (see \cite{Han-Wang-Wang,Teb1,Teb2}). It is also worthwhile to notify that, in \cite{Teb1,Teb2} the author consider another fluid–structure interaction model, where the structure is modeled by an Euler-Bernoulli beam and the fluid by a weakly degenerate parabolic equation and he proved that the underlying semigroup is both exponentially stable, and of Gevrey class. Concerning the thermoplate  equation with degeneracy occurring in the heat component, as far as we know, there is only one work who treated the problem of stability for a such model (see \cite{Ata-Kam}) where the authors analyze the evolution of the energy density of a family of solutions for the higher dimension case by means of a method based on the use of microlocal defect measures. In this paper and unlike the works of Tebou \cite{Teb1,Teb2}, we are focusing not only on the problem of weak degeneracy but also to the problem of strong degeneracy as well. We shall prove that both problems are well-posed by means of semigroup theory and we prove a uniform decay rate of the energy for both systems using frequency domain method.

\medskip

The reminder of the paper is organized as follow: Section \ref{WPTE} deals with the well-posedness of the problems \eqref{ITE4} and \eqref{ITE6}. Section \ref{PSTE} is devoted to the study of strong stability and uniform stability for both systems. 
\section{Well-posedness}\label{WPTE}
Let $\ds\mathcal{H}=H_{0}^{1}(0,1)\times L^{2}(0,1)\times L^{2}(0,1)$ be the Hilbert space endowed with the inner product define for $(u,v,\theta),\,(\tilde{u},\tilde{v},\tilde{\theta})\in\mathcal{H}$ by
$$ 
\left\langle (u,v,\theta),(\tilde{u},\tilde{v},\tilde{\theta})\right\rangle_{\mathcal{H}}=\int_{0}^{1}u'(x)\,.\,\overline{\tilde{u}'}(x)\,\ud x+\int_{0}^{1}v(x)\,.\,\overline{\tilde{v}}(x)\,\ud x+\int_{0}^{1}\theta(x)\,.\,\overline{\tilde{\theta}}(x)\,\ud x.
$$
By setting $U(t)=(u(t),\dot{u}(t),\theta(t))$ and $U^{0}=(u^{0},u^{1},\theta^{0})$ we can rewrite system \eqref{ITE1} as a first order differential equation as follows
\begin{equation}\label{WPTE1}
\dot{U}(t)=\mathcal{A}U(t),\qquad U(0)=U^{0}\in\mathcal{D}(\mathcal{A}),
\end{equation}
where
$$
\mathcal{A}(u,v,\theta)=(v,u''-\kappa\theta',(a\theta')'-\kappa v'),
$$
with
$$
\mathcal{D}(\mathcal{A})=\left\{(u,v,\theta)\in\mathcal{H}:\;(u,v)\in\left(H_{0}^{1}(0,1)\right)^{2},\;\theta\in H_{a}^{1}(0,1),\;u''-\kappa\theta'\in L^{2}(0,1),\;a\theta'\in H^{1}(0,1)\right\}. 
$$
where
$$
H_{a}^{1}(0,1)=\left\{\theta\in L^{2}(0,1):\;\theta \text{ is absolutely continuous in }[0,1],\; \sqrt{a}\,\theta'\in L^{2}(0,1),\;\theta(0)=\theta(1)=0\right\}
$$
in the case of (WDP) and
$$
H_{a}^{1}(0,1)=\left\{\theta\in L^{2}(0,1):\;\theta \text{ is locally absolutely continuous in }(0,1],\; \sqrt{a}\,\theta'\in L^{2}(0,1),\;\theta(1)=0\right\}
$$
in the case of (SDP).

\medskip

We state a preliminary result that will also be used for the well-posedness of the problem.
\begin{lem}\label{WPTE9}
Let $\theta$ be such that $(a\theta)\in H_{\ell}^{1}(0,1) := \left\{u \in H^1(0,1), \, u(0) =0 \right\},$ then
$$
|a(x)\theta(x)|\leq \sqrt{x}\|(a\theta)'\|_{L^{2}(0,1)}\qquad\forall\,x\in[0,1]
$$
and if $a\theta'\in H^{1}(0,1)$ such that $(a\theta')(0)=0$,
$$
|a(x)\theta'(x)|\leq \sqrt{x}\|(a\theta')'\|_{L^{2}(0,1)}\qquad\forall\,x\in[0,1].
$$
\end{lem}
\begin{proof}
Since $(a\theta)\in H_{\ell}^{1}(0,1)$ and, then by Cauchy–Schwarz inequality
$$
|(a\theta)(x)|=\left|\int_{0}^{x}(a\theta)'(s)\,\ud s\right|\leq \sqrt{x}\|(a\theta)'\|_{L^{2}(0,1)}.
$$
Now, if $a\theta'\in H^{1}(0,1)$ and $(a\theta')(0)=0$ then similar arguments yield to
$$
|(a\theta')(x)|=\left|\int_{0}^{x}(a\theta')'(s)\,\ud s\right|\leq \sqrt{x}\|(a\theta')'\|_{L^{2}(0,1)}.
$$
This completes the proof.
\end{proof}
\begin{lem}\label{WPTE10}
We suppose that $a\in\mathcal{C}^{1}([0,1])$ such that $a>0$ in $(0,1]$, $\ds\frac{1}{a}\notin L^{1}(0,1)$ and $\ds\frac{1}{\sqrt{a}}\in L^{1}(0,1)$. Let $\theta\in L^{2}(0,1)$ be a locally absolutely continuous in $(0,1]$ such that $(a\theta')\in H^{1}(0,1)$ then we have the equivalence between the following two statements
\begin{itemize}
	\item[i)] $\theta\in H_{a}^{1}(0,1)$.
	\item[ii)] $(a\theta)\in H_{0}^{1}(0,1)$ and $(a\theta')(0)=0$.
\end{itemize}
Besides, if one of the statements holds true then
$$
\int_{0}^{1}(a\theta')'\theta\,\ud s=-\int_{0}^{1}a|\theta'|^{2}\,\ud s.
$$
\end{lem}
\begin{proof}
Let's first prove that ii) implies i). Assume that $\theta$ such that ii) holds true. To prove i) we only need to prove that $\sqrt{a}\theta'\in L^{2}(0,1)$.
\\
For $x\in(0,1)$, we write
$$
\int_{x}^{1}(a\theta')'\theta\,\ud s=(a\theta')(1)\theta(1)-a(x)\theta'(x)\theta(x)-\int_{x}^{1}a|\theta'|^{2}\,\ud s.
$$
Since $(a\theta')\in H^{1}(0,1)$ and $\theta(1)=0$, we have $(a\theta')(1)\theta(1)=0$. Hence,
$$
a(x)\theta'(x)\theta(x)=-\int_{x}^{1}(a\theta')'\theta\,\ud s-\int_{x}^{1}a|\theta'|^{2}\,\ud s.
$$
Since $(a\theta')\in H^{1}(0,1)$ and $\theta\in L^{2}(0,1)$ then $(a\theta')'\theta\in L^{1}(0,1)$. Thus, there exists $\ell\in[-\infty,+\infty)$ such that 
$$
(a\theta')(x)\theta(x)\underset{x\to 0}{\longrightarrow} \ell.
$$
Let's suppose that $\ell\neq0$. Then there exists $C>0$ such that, for $x$ small enough $|(a\theta')(x)\theta(x)|\geq C$. Using Lemma \ref{WPTE9}, there exists $C'>0$ such that for $x$ small enough $\ds|\theta(x)|\geq\frac{C'}{\sqrt{x}}$, which implies that $\theta\notin L^{2}(0,1)$ and contradict the fact that $\theta\in L^{2}(0,1)$. Hence, $\ell=0$ and we obtain that
$$
\int_{0}^{1}(a\theta')'\theta\,\ud s=-\int_{0}^{1}a|\theta'|^{2}\,\ud s.
$$
In particular we deduce that $\sqrt{a}\,\theta'\in L^{2}(0,1)$.

\medskip

Now let's prove that i) implies ii). Let's $\theta\in H_{a}^{1}(0,1)$ such that $(a\theta')\in H^{1}(0,1)$. To prove $\theta$ follows the statement of ii) we need to prove $(a\theta)\in H_{0}^{1}(0,1)$ and $(a\theta')(0)=0$. Since $(a\theta')\in H^{1}(0,1)$ there exists $\ell\in\R$ such that $(a\theta')\underset{x\to 0}{\longrightarrow}\ell$. If $\ell\neq 0$ then for $x$ small enough 
$$
a(x)|\theta'(x)|^{2}\geq \frac{\ell^{2}}{2a(x)}\notin L^{1}(0,1).
$$
Hence, $(\sqrt{a}\theta')\notin L^{2}(0,1)$. This implies $\ell=0$ and $(a\theta')(0)=0$. This implies $(a\theta)\in H^{1}(0,1)$ since $(a\theta)'=a'\theta+a\theta'\in L^{2}(0,1)$. Indeed, $a'\theta\in L^{2}(0,1)$, since $a\in \mathcal{C}^{1}([0,1])$ and we also have $(a\theta)(1)=0$, since $\theta(1)=0$ and $a(1)>0$.

\medskip

It remains to prove that $(a\theta)(0)=0$. Since $(a\theta)\in H^{1}(0,1)$, there exists $\ell\in\R$ such that $(a\theta)\underset{x\to 0}{\longrightarrow}\ell$. If $\ell\neq 0$, then $\theta\notin L^{2}(0,1)$. Indeed, it would implies for $x$ small enough
$$
|\theta(x)|^{2}\underset{0}{\sim}\frac{\ell^{2}}{a(x)^{2}}\geq\frac{C}{a(x)}\notin L^{1}(0,1)
$$
for some $C>0$. Thus we obtain $\ell=0$. This achieve the proof.
\end{proof}
\begin{rem}
Thanks to Lemma \ref{WPTE10} we can write $\mathcal{D}(\mathcal{A})$ in the case of (SDP) as follows:
\begin{align*}
\mathcal{D}(\mathcal{A})=\Big\{(u,v,\theta)\in\mathcal{H}:\; (u,v)\in\left(H_{0}^{1}(0,1)\right)^{2},\;\theta \text{ is locally absolutely continuous in }(0,1],
\\
u''-\kappa\theta'\in L^{2}(0,1),\;(a\,\theta)\in H_{0}^{1}(0,1),\;(a\theta')\in H^{1}(0,1),\; (a\theta')(0)=0\Big\}.
\end{align*}
\end{rem}
\begin{prop}\label{WPTE2}
The operator $\mathcal{A}$ is dissipative in $\mathcal{H}$ and we have
\begin{equation}\label{WPTE14}
\re\left\langle\mathcal{A}\left(\begin{array}{c}
u
\\
v
\\
\theta
\end{array}\right),\left(\begin{array}{c}
u
\\
v
\\
\theta
\end{array}\right)\right\rangle_{\mathcal{H}}=-\int_{0}^{1}a(s)|\theta'(s)|^{2}\,\ud s,\quad\forall\,\left(\begin{array}{c}
u
\\
v
\\
\theta
\end{array}\right)\in\mathcal{D}(\mathcal{A}).
\end{equation}
\end{prop}
\begin{proof}
Let $\ds U=\left(\begin{array}{c}
u
\\
v
\\
\theta
\end{array}\right),\,\widetilde{U}=\left(\begin{array}{c}
\tilde{u}
\\
\tilde{v}
\\
\tilde{\theta}
\end{array}\right)\in \mathcal{D}(\mathcal{A})$. For all $x_{1},\,x_{2}\in(0,1)$ we have
\begin{equation}\label{WPTE11}
\int_{x_{1}}^{x_{2}}(a\theta')'\tilde{\theta}\,\ud s=(a\theta')(x_{2})\tilde{\theta}(x_{2})-(a\theta')(x_{1})\tilde{\theta}(x_{1})-\int_{x_{1}}^{x_{2}}(a\theta')(s)\tilde{\theta}'(s)\,\ud s.
\end{equation}
We need first to prove that
\begin{equation}\label{WPTE13}
a(x)\theta'(x)\tilde{\theta}(x)\underset{x\to 0}{\longrightarrow}0\quad\text{ and }\quad a(x)\theta'(x)\tilde{\theta}(x)\underset{x\to 1}{\longrightarrow}0.
\end{equation}
The second statement is guarantee since $(a\theta')\in H^{1}(0,1)$ and $\tilde{\theta}(1)=0$ so it still only to prove the first one. To do so, we distinguish the two cases:

\medskip

\textbf{Case (WDP):} Since $(a\theta')\in H^{1}(0,1)$ and $\tilde{\theta}(0)=0$ and $\tilde{\theta}(0)=0$ we deduce that $(a\theta')(0)\tilde{\theta}(0)=0$.

\medskip

\textbf{Case (SDP):} From \eqref{WPTE11} and the analysis that follows, we have
\begin{equation}\label{WPTE12}
\int_{x}^{1}(a\theta')'\tilde{\theta}\,\ud s=-(a\theta')(x)\tilde{\theta}(x)-\int_{x}^{1}(a\theta')(s)\tilde{\theta}'(s)\,\ud s.
\end{equation}
Since $(a\theta')\in H^{1}(0,1)$ and $\tilde{\theta},\,\sqrt{a}\theta'$ and $\sqrt{a}\tilde{\theta}'$ belong to $\in L^{2}(0,1)$, we deduce that $(a\theta')'\tilde{\theta}$ and $a\theta'\tilde{\theta}'$ belong to $L^{2}(0,1)$. Consequently, from \eqref{WPTE12} there exists $\ell\in\R$ such that $(a\theta')(x)\tilde{\theta}(x)\underset{x\to 0}{\longrightarrow}\ell$.

\medskip

If $\ell\neq 0$, then it would imply that $\tilde{\theta}\notin L^{2}(0,1)$. Indeed, for some $c>0$ and for $x$ small enough that
$$
|\tilde{\theta}(x)|\geq\frac{\ell}{2|a(x)\theta'(x)|}\geq \frac{c}{\sqrt{x}}\notin L^{2}(0,1).
$$
This makes $\ell =0$ and therefore $(a\theta')(0)\tilde{\theta}(0)=0$ and \eqref{WPTE13} is now proven in both cases. 

\medskip

Now with this we can perform the following calculations
\begin{align}\label{WPTE24}
\left\langle\mathcal{A}U,\widetilde{U}\right\rangle_{\mathcal{H}}&=\int_{0}^{1}v'\overline{\tilde{u}'}\,\ud x+\int_{0}^{1}\left(u''-\kappa\theta'\right)\overline{\tilde{v}}\,\ud x+\int_{0}^{1}\left((a\theta')'-\kappa v'\right)\overline{\tilde{\theta}}\,\ud x\nonumber
\\
&=\int_{0}^{1}v'\overline{\tilde{u}'}\,\ud x-\int_{0}^{1}u'\overline{\tilde{v}'}\,\ud x+\kappa\int_{0}^{1}\theta\overline{\tilde{v}'}\,\ud x-\int_{0}^{1}a\theta'\overline{\tilde{\theta}'}\,\ud x-\kappa\int_{0}^{1}v'\overline{\tilde{\theta}}\,\ud x.
\end{align}
Thus \eqref{WPTE14} follows by taking the real part of $U=\widetilde{U}$ in the last estimate which prove that $\mathcal{A}$ is dissipative and this achieve the proof.
\end{proof}
\begin{prop}
The operator $\mathcal{A}$ is closed in $\mathcal{H}$.
\end{prop}
\begin{proof}
Consider a sequence $(u_{n},v_{n},\theta_{n})$ that belongs to $\mathcal{D}(\mathcal{A})$ such that
\begin{equation}\label{WPTE15}
\left(\begin{array}{c}
u_{n}
\\
v_{n}
\\
\theta_{n}
\end{array}\right)\underset{n\to+\infty}{\longrightarrow}\left(\begin{array}{c}
u
\\
v
\\
\theta
\end{array}\right)\text{ and }\mathcal{A}\left(\begin{array}{c}
u_{n}
\\
v_{n}
\\
\theta_{n}
\end{array}\right)=\left(\begin{array}{c}
v_{n}
\\
u_{n}''-\kappa\theta_{n}'
\\
(a\theta_{n}')'-\kappa v_{n}'
\end{array}\right)\underset{n\to+\infty}{\longrightarrow}\left(\begin{array}{c}
f
\\
g
\\
h
\end{array}\right)\in \mathcal{H}.
\end{equation}
To prove that $\mathcal{A}$ is closed we need to prove that $(u,v,\theta)\in\mathcal{D}(\mathcal{A})$ and $(v,u''-\kappa\theta',(a\theta')'-\kappa v')=(f,g,h)$. First let notes that $v_{n}\underset{n\to+\infty}{\longrightarrow}v$ in $L^{2}(0,1)$ and $v_{n}\underset{n\to+\infty}{\longrightarrow}f$ in $H_{0}^{1}(0,1)$ then $v\in H_{0}^{1}(0,1)$ and $v=f$. This implies that
$$
u_{n}''-\kappa\theta_{n}'\underset{n\to+\infty}{\longrightarrow}g\text{ in } L^{2}(0,1)\quad\text{ and }\quad (a\theta_{n}')'\underset{n\to+\infty}{\longrightarrow}\kappa f'+h\text{ in } L^{2}(0,1).
$$
We distinguish two cases:

\medskip

\textbf{Case (WDP):} For all $x\in[0,1]$ and $n,\,p\in\N$, we write using $\theta_{n}(0)=\theta_{p}(0)=0$,
\begin{equation}\label{WPTE16}
\theta_{n}(x)-\theta_{p}(x)=\int_{0}^{x}\theta_{n}'(s)-\theta_{p}'(s)\,\ud s.
\end{equation}
Since $\ds\frac{1}{a}\in L^{1}(0,1)$, we deduce for all $x\in (0,1)$,
\begin{align*}
|\theta_{n}(x)-\theta_{p}(x)|&\leq\left(\int_{0}^{x}\frac{\ud s}{a}\right)^{\frac{1}{2}}\left(\int_{0}^{x}a|\theta_{n}'|^{2}+a|\theta_{p}'|^{2}-2a \, Re \,[\theta_{n}'\theta_{p}']\,\ud s\right)^{\frac{1}{2}}
\\
&\leq\left(\int_{0}^{1}\frac{\ud s}{a}\right)^{\frac{1}{2}}\left(\int_{0}^{1}(a\theta_{n}')'\overline{\theta}_{n}+(a\theta_{p}')'\overline{\theta}_{p}-2 \, Re \, [\theta_{n}(a\theta_{p}')']\,\ud s\right)^{\frac{1}{2}},
\end{align*}
where we have used \eqref{WPTE16}. As $p$ goes to the infinity
\begin{align*}
|\theta_{n}(x)-\theta(x)|&\leq\left(\int_{0}^{1}\frac{\ud s}{a}\right)^{\frac{1}{2}}\left(\int_{0}^{1}(a\theta_{n}')'\overline{\theta}_{n}+(\kappa f'+h)\overline{\theta}-2 \, Re \, [\theta_{n}(\kappa f'+h)]\,\ud s\right)^{\frac{1}{2}}\quad\forall\,x\in[0,1].
\end{align*}
Hence,
$$
\sup_{x\in[0,1]}|\theta_{n}(x)-\theta(x)|\leq\left\|\frac{1}{a}\right\|_{L^{1}(0,1)}^{\frac{1}{2}}\left(\int_{0}^{1}(a\theta_{n}')'\overline{\theta}_{n}+(\kappa f'+h)\overline{\theta}-2 \, Re \, [\theta_{n}(\kappa f'+h)]\,\ud s\right)^{\frac{1}{2}}\underset{n\to+\infty}{\longrightarrow}0.
$$
We conclude that $\theta_{n}\underset{n\to+\infty}{\longrightarrow}\theta$ uniformly on $[0,1]$. In particular, using $\theta_{n}(0)=\theta_n(1)=0$ for all $n\in\N$ it implies that $\theta(0)=\theta(1)=0$.

Let $\varphi\in\mathcal{C}^{\infty}([0,1])$ be positive on $(0,1]$ such that $\varphi(0)=0$. Notice that $(a\theta_{n}')(0)\varphi(0)=0$ and $(a\theta_{n})(0)\varphi'(0)=0$ we deduce, for all $x\in[0,1]$,
$$
\int_{0}^{x}(a\theta_{n}')'\varphi\,\ud s=\int_{0}^{x}\theta_{n}(a\varphi')'\,\ud s-(a\theta_{n})(x)\varphi'(x)+(a\theta_{n}')(x)\varphi(x).
$$
Hence,
$$
(a\theta_{n}')(x)\varphi(x)=\int_{0}^{x}(a\theta_{n}')'\varphi\,\ud s-\int_{0}^{x}\theta_{n}(a\varphi')'\,\ud s+(a\theta_{n})(x)\varphi'(x).
$$
Since $(a\theta_{n}')'\underset{n\to+\infty}{\longrightarrow}h+\kappa f'$ in $L^{2}(0,1)$ and $\theta_{n}\underset{n\to+\infty}{\longrightarrow}\theta$ uniformly on $[0,1]$, it implies that $(a\theta_{n}')\varphi$ converge uniformly on $[0,1]$. In particular, $\theta'\in H^{1}(\epsilon,1)$ and $\theta_{n}'$ converge uniformly to $\theta'$ on $[\varepsilon,1]$ for all $\varepsilon>0$.

\medskip

Now we write, for all $x\in[0,1]$ 
\begin{equation}\label{WPTE17}
\int_{0}^{x}(a\theta_{n}')'\,\ud s=(a\theta_{n}')(x)-(a\theta_{n}')(0).
\end{equation}
Since $(a\theta_{n}')'\underset{n\to+\infty}{\longrightarrow}h+\kappa f'$ in $L^{2}(0,1)$ and $\theta_{n}'(x)\underset{n\to+\infty}{\longrightarrow}\theta'(x)$ for all $x>0$, we deduce that $(a\theta_{n}')(0)$ has a limit as $n$ goes to $\infty$. We denote by $\ell$ this limit. Passing to the limit in \eqref{WPTE17} yields to
\begin{equation}\label{WPTE19}
(a\theta')(x)=\int_{0}^{x}(a\theta')'\,\ud s-\ell,\quad\forall\, x\in(0,1].
\end{equation}
This implies $(a\theta')\in H^{1}(0,1)$.

\medskip

Now let us prove that $\sqrt{a}\theta'\in L^{2}(0,1)$. From \eqref{WPTE12}, it turns out that
$$
\int_{0}^{1}(a\theta_{n}')'\overline{\theta_{n}}\,\ud s=-\int_{0}^{1}a|\theta_{n}'|^{2}\,\ud s,\quad\forall\,n\in\N.
$$
Hence,
$$
\int_{0}^{1}a|\theta_{n}'|^{2}\,\ud s\underset{n\to+\infty}{\longrightarrow}-\int_{0}^{1}(h+\kappa f')\theta\,\ud s.
$$
In the other hand, the sequence $(a|\theta_{n}'|^{2})$ converge almost everywhere in $(0,1)$ to $a|\theta|^{2}$. Thus, Fatou's lemma implies
$$
\int_{0}^{1}a|\theta'|^{2}\,\ud s\leq \liminf_{n\to+\infty}\int_{0}^{1}a|\theta_{n}'|^{2}\,\ud s<+\infty
$$
and consequently, we obtain that $\sqrt{a}\theta'\in L^{2}(0,1)$.

\medskip

To prove that $(u,v,\theta)\in\mathcal{D}(\mathcal{A})$, it remains to prove that $\theta$ is absolutely continuous on $[0,1]$, $u,\,v\in H_{0}^{1}(0,1)$ and $(u''-\kappa\theta')\in L^{2}(0,1)$. First $\ds\theta'=\frac{1}{a}(a\theta')\in L^{2}(0,1)$ since $\ds\frac{1}{a}\in L^{1}(0,1)$ and $(a\theta')\in H^{1}(0,1)\subset L^{\infty}(0,1)$. Next we have to verify that
$$
\forall\,x_{1},\,x_{2}\in[0,1],\quad \theta(x_{2})-\theta(x_{1})=\int_{x_{1}}^{x_{2}}\theta'(s)\,\ud s.
$$
Notice that, $\theta(0)=0$, it is sufficient to prove
\begin{equation}\label{WPTE18}
\forall\,x\in[0,1],\quad \theta(x)=\int_{0}^{x}\theta'(s)\,\ud s.
\end{equation}
Let $x$ be in $(0,1]$ and we write
$$
\theta_{n}(x)=\int_{0}^{x}\theta_{n}(s)\,\ud s=\int_{0}^{\varepsilon}\theta_{n}'(s)\,\ud s+\int_{\varepsilon}^{x}\theta_{n}'(s)\,\ud s\quad\forall\,n\in\N.
$$
Hence,
\begin{align*}
\left|\theta_{n}(x)-\int_{\varepsilon}^{x}\theta_{n}'(s)\,\ud s\right|&=\left|\int_{0}^{\varepsilon}\theta_{n}'(s)\,\ud s\right|
\\
&\leq\left(\int_{0}^{\varepsilon}\frac{\ud s}{a(s)}\right)^{\frac{1}{2}}\left(\int_{0}^{1}a(s)|\theta_{n}'(s)|^{2}\,\ud s\right)^{\frac{1}{2}}.
\end{align*}
Since $\theta_{n}\underset{n\to+\infty}{\longrightarrow}\theta$ in $\mathcal{C}^{1}([\varepsilon,1])$ and in $L^{2}(0,1)$ for all $\varepsilon>0$, we can pass the limit to $+\infty$ in the above relation and obtain
$$
\left|\theta(x)-\int_{\varepsilon}^{x}\theta'(s)\,\ud s\right|\leq C\left(\int_{0}^{\varepsilon}\frac{\ud s}{a(s)}\right)^{\frac{1}{2}}\left(\int_{0}^{1}a(s)|\theta'(s)|^{2}\,\ud s\right)^{\frac{1}{2}}\underset{\varepsilon\to 0}{\longrightarrow}0,
$$
then \eqref{WPTE18} follows, where we used the fact that $\ds\frac{1}{a}\in L^{1}(0,1)$ and $\theta'\in L^{2}(0,1)$.

\medskip

It's clear that $u,\,v\in H_{0}^{1}(0,1)$ since $u_{n}\underset{n\to+\infty}{\longrightarrow} u$  and $v_{n}\underset{n\to+\infty}{\longrightarrow}f$ in $H_{0}^{1}(0,1)$ and $v_{n}\underset{n\to+\infty}{\longrightarrow}v$ in $L^{2}(0,1)$. Since $u_{n}''-\kappa\theta_{n}'\underset{n\to+\infty}{\longrightarrow}g$ in $L^{2}(0,1)$ and $\theta_{n}\underset{n\to+\infty}{\longrightarrow}\kappa\theta$ in $\mathcal{C}^{1}([\varepsilon,1])$ for all $\varepsilon>0$, then $u_{n}''\underset{n\to+\infty}{\longrightarrow}g-\kappa\theta'$ in $L^{2}(\varepsilon,1)$. Besides, we have $u_{n}\underset{n\to+\infty}{\longrightarrow}u$ in $H_{0}^{1}(0,1)$ then $u_{n}\underset{n\to+\infty}{\longrightarrow}u$ in $H^{2}(\varepsilon,1)$ and we have $u''=g-\kappa\theta'$ almost everywhere in $(0,1)$ which gives $u''-\kappa\theta'=g\in L^{2}(0,1)$. From \eqref{WPTE19} we have $(a\theta')'=h+\kappa f'$. With this we have proved that $(u,v,\theta)\in\mathcal{D}(\mathcal{A})$ and $\mathcal{A}(u,v,\theta)=(f,g,h)$ and achieve the first part of the proof.

\medskip

\textbf{Case (SDP):} We define
\begin{equation}\label{WPTE20}
w_{n}(x)=\int_{0}^{x}(a\theta_{n}')'(s)\,\ud s=(a\theta_{n}')(x)\quad\text{ and }\quad w(x)=\int_{0}^{x}(h+\kappa f')(s)\,\ud s.
\end{equation}
Since $(a\theta_{n}')'\underset{n\to+\infty}{\longrightarrow}h+\kappa f'$ in $L^{2}(0,1)$, we deduce that $w_{n}=a\theta_{n}'\underset{n\to+\infty}{\longrightarrow} w$ uniformly on $[0,1]$. Thus, for all $\varepsilon>0$,
$$
\theta_{n}'=\frac{w_{n}}{a}\underset{n\to+\infty}{\longrightarrow}\frac{w}{a}\quad\text{ uniformly on }[\varepsilon,1].
$$
In one hand $\theta_{n}\underset{n\to+\infty}{\longrightarrow}\theta$ in $L^{2}(0,1)$ and in another hand since $\theta_{n}(1)=0$ for all $n\in\N$ and then by writing
\begin{equation}\label{WPTE22}
\theta_{n}(x)=-\int_{x}^{1}\theta_{n}'(s)\,\ud s
\end{equation}
we deduce that $\ds\theta_{n}\underset{n\to+\infty}{\longrightarrow}-\int_{x}^{1}\frac{w(s)}{a(s)}\,\ud s$ uniformly on $[\varepsilon,1]$. Hence, $\ds\theta_{n}\underset{n\to+\infty}{\longrightarrow}\theta$ in $\mathcal{C}^{1}([\varepsilon,1])$. Passing to the limit in $w_{n}(x)=a(x)\theta_{n}'(x)$, we obtain $a(x)\theta'(x)=w(x)$ for every $x\in(0,1]$. This implies that $a\theta'\in H^{1}(0,1)$ and $a\theta\in H^{1}(0,1)$ since $(a\theta)'=a'\theta+a\theta'\in L^{2}(0,1)$ with the fact that $a'\theta\in L^{2}(0,1)$ and where we recall that $a\in\mathcal{C}^{1}([0,1])$. Moreover, since $(a\theta_{n})(1)=0$ we also have $(a\theta)(1)=0$. From Lemma \ref{WPTE9} and by passing to the limit in \eqref{WPTE20}, we deduce that for all $x\in(0,1]$,
\begin{equation*}
|a(x)\theta(x)|\leq \sqrt{x}\,\|w+a'\theta\|_{L^{2}(0,1)}\quad\text{ and }\quad|a(x)\theta'(x)|\leq \sqrt{x}\,\|h+\kappa f'\|_{L^{2}(0,1)}.
\end{equation*}
Hence, $(a\theta)(0)=0$ and $(a\theta')(0)=0$.

\medskip

Now we prove that $\theta$ is locally absolutely continuous on $(0,1]$. First, since $\ds\frac{1}{a}\in L_{\mathrm{loc}}^{1}((0,1])$ and $a\theta'\in H^{1}(0,1)\subset L^{\infty}(0,1)$ we deduce that $\ds\theta'=\frac{1}{a}(a\theta')\in L^{1}_{\mathrm{loc}}((0,1])$. Next we need to verify that 
$$
\theta(x_{2})-\theta(x_{1})=\int_{x_{1}}^{x_{2}}\theta'(s)\,\ud s\quad\forall\,x_{1},\,x_{2}\in(0,1].
$$
Notice that $\theta(1)=0$ then it is sufficient to prove that
\begin{equation}\label{WPTE21}
\theta(x)=-\int_{x}^{1}\theta'(s)\,\ud s\quad\forall\,x\in(0,1].
\end{equation}
Since $\theta_{n}\underset{n\to+\infty}{\longrightarrow}\theta$ in $\mathcal{C}^{1}([\varepsilon,1])$ for all $\varepsilon>0$ then we can passe to the limit in \eqref{WPTE22} and obtain \eqref{WPTE21}. Moreover, we have 
$$
(a\theta')(x)=w(x)=\int_{0}^{x}(h+\kappa f')(s)\,\ud s\quad\forall\,x\in[0,1]
$$
which implies that $(a\theta')'=h+\kappa f'$. To finish the proof we proceed as the last part of the first case in order to prove that ${}^{t}(u,v,\theta)\in\mathcal{D}(\mathcal{A})$ and $\mathcal{A}\,{}^{t}(u,v,\theta)={}^{t}(f,g,h)$.
\end{proof}
\begin{prop}\label{WPTE3}
The adjoint operator of $\mathcal{A}$ is given by
\begin{equation}\label{WPTE23}
\mathcal{D}(\mathcal{A}^{*})=\mathcal{D}(\mathcal{A}),\quad \mathcal{A}^{*}\left(\begin{array}{c}
u
\\
v
\\
\theta
\end{array}\right)=\left(\begin{array}{c}
-v
\\
-u''+\kappa\theta'
\\
(a\theta')'+\kappa v'
\end{array}\right)\quad\forall\,\left(\begin{array}{c}
u
\\
v
\\
\theta
\end{array}\right)\in\mathcal{D}(\mathcal{A}).
\end{equation}
\end{prop}
\begin{proof}
We recall that $\mathcal{A}^{*}$ is defined 
$$
\forall\,\widetilde{U}={}^{t}(\tilde{u},\tilde{v},\tilde{\theta})\in\mathcal{D}(\mathcal{A}^{*}),\;\forall\,U={}^{t}(u,v,\theta)\in\mathcal{D}(\mathcal{A})\quad \left\langle\mathcal{A}U,\tilde{U}\right\rangle=\left\langle U,\mathcal{A}^{*}\tilde{U}\right\rangle.
$$
where,
$$
\mathcal{D}(\mathcal{A}^{*})=\Big\{U_{*}\in\mathcal{H}:\quad\exists\,C>0,\;\forall\,U\in\mathcal{D}(\mathcal{A}),\;|\left\langle\mathcal{A}U,U_{*}\right\rangle|\leq C\|U\|_{\mathcal{H}}\Big\}.
$$

First we show that $\mathcal{D}(\mathcal{A})\subset \mathcal{D}(\mathcal{A}^{*})$ and $\mathcal{A}^{*}$ restricted to $\mathcal{D}(\mathcal{A})$ is given as in \eqref{WPTE23}. Indeed, for all $U={}^{t}(u,v,\theta)$ and $\widetilde{U}={}^{t}(\tilde{u},\tilde{v},\tilde{\theta})$ and from \eqref{WPTE24} we have
\begin{align*}
\left\langle\mathcal{A}U,\widetilde{U}\right\rangle&=-\int_{0}^{1}u'\overline{\tilde{v}'}\,\ud x+\int_{0}^{1}v\left(\kappa\overline{\tilde{\theta}}-\overline{\tilde{u}''}\right)\,\ud x+\int_{0}^{1}\theta\left(\left((a\overline{\tilde{\theta}'}\right)'+\kappa\overline{\tilde{v}'}\right)\,\ud x.
\end{align*}
Thus, our first expected result is shown.

\medskip

Next, we prove that $\mathcal{D}(\mathcal{A}^{*})\subset\mathcal{D}(\mathcal{A})$, and for this aim we only need to prove that $\mathcal{D}(\mathcal{A}^{*})=\mathcal{D}(-\mathcal{A}^{*}+J)\subset \mathcal{D}(-\mathcal{A}+J)=\mathcal{D}(\mathcal{A})$ where $J$ is the bounded operator defined in $\mathcal{H}$ by
$$
J\left(\begin{array}{c}
u
\\
v
\\
\theta
\end{array}\right)=\left(\begin{array}{r}
-u
\\
-v
\\
\theta
\end{array}\right).
$$
We note that $H_{a}^{1}(0,1)$ is a Hilbert space with the scalar product
$$
\langle\theta,\phi\rangle_{H_{a}^{1}(0,1)}=\int_{0}^{1}\theta(s)\overline{\phi}(s)+a\theta'(s)\overline{\phi'}(s)\,\ud s. 
$$

We shall first prove that $(-\mathcal{A}^{*}+J)_{|\mathcal{D}(\mathcal{A})}$ is surjective. So let $(f,g,h)\in \mathcal{H}$ then with the assumption made on $\kappa$ there exists $(u_{*},\theta_{*})\in H_{0}^{1}(0,1)\times H_{a}^{1}(0,1)$ such that for all $(\varphi,\phi)\in H_{0}^{1}(0,1)\times H_{a}^{1}(0,1)$
\begin{align*}
\langle u_{*}',\varphi'\rangle_{L^{2}(0,1)}+\langle\theta_{*},\phi\rangle_{H_{a}^{1}(0,1)}+\langle u_{*},\varphi\rangle_{L^{2}(0,1)}-\langle\kappa u_{*}',\phi\rangle_{L^{2}(0,1)}-\langle\kappa\theta_{*},\varphi'\rangle_{L^{2}(0,1)}+\langle a\theta_{*}',\phi'\rangle_{L^{2}(0,1)}=
\\
\langle-f-g,\varphi\rangle_{L^{2}(0,1)}+\langle\kappa f'+h,\phi\rangle_{L^{2}(0,1)}
\end{align*}
as the left hand side is a inner product equivalent to the one of $H^1\times H^1_a$. This implies that for all $(\varphi,\phi)\in\left(D(0,1)\right)^{2}$, we get
\begin{align*}
\langle -u_{*}''+u_{*}+\kappa\theta_{*}',\varphi\rangle_{D'(0,1)}+\langle-(a\theta_{*}')'+\theta_{*}-\kappa u_{*}',\phi\rangle_{D'(0,1)}=\langle-f-g,\varphi\rangle_{D'(0,1)}+\langle h+\kappa f',\phi\rangle_{D'(0,1)}.
\end{align*}
We obtain
$$
\left\{\begin{array}{l}
-u_{*}''+u_{*}+\kappa\theta_{*}'=-f-g
\\
-(a\theta_{*}')'+\theta_{*}-\kappa u_{*}'=h+\kappa f'
\end{array}\right.
$$
or equivalently, by setting $v_{*}=u_{*}+f\in H_{0}^{1}(0,1)$
\begin{equation}\label{WPTE26}
\left\{\begin{array}{l}
v_{*}-u_{*}=f
\\
u_{*}''-\kappa\theta_{*}'-v_{*}=g
\\
-(a\theta_{*}')'-\kappa v_{*}'+\theta_{*}=h
\end{array}\right.
\end{equation}
which implies that for all ${}^{t}(f,g,h)\in\mathcal{H}$ there exists a unique $U_{*}={}^{t}(u_{*},v_{*},\theta_{*})\in\mathcal{D}(\mathcal{A})$ such that
$$
(-\mathcal{A}^{*}+J)U_{*}=(-\mathcal{A}^{*}+J)\left(\begin{array}{c}
u_{*}
\\
v_{*}
\\
\theta_{*}
\end{array}\right)=\left(\begin{array}{c}
f
\\
g
\\
h
\end{array}\right).
$$
Then $(-\mathcal{A}^{*}+J)$ is surjective is deduced consequently.

\medskip

Now let $\widetilde{U}_{*}\in\mathcal{D}(\mathcal{A}^{*})$ and let's show that $\widetilde{U}_{*}$ belongs to $\mathcal{D}(\mathcal{A})$. From the definition of $\mathcal{D}(\mathcal{A}^{*})$ and the Riesz representation theorem there exists a unique $W\in\mathcal{H}$ such that
$$
\left\langle(-\mathcal{A}+J)\Phi,\widetilde{U}_{*}\right\rangle=\left\langle\Phi,W\right\rangle\quad\forall\,\Phi={}^{t}(\varphi,\psi,\phi)\in\mathcal{D}(\mathcal{A}).
$$
Since $(-\mathcal{A}^{*}+J)_{|\mathcal{D}(\mathcal{A})}$ is surjective then we set $U_{*}={}^{t}(u_{*},v_{*},\theta_{*})\in\mathcal{D}(\mathcal{A})$ such that $W=(-\mathcal{A}^{*}+J)U_{*}$ and we obtain
$$
\left\langle(-\mathcal{A}+J)\Phi,\widetilde{U}_{*}\right\rangle=\left\langle\Phi,(-\mathcal{A}^{*}+J)U_{*}\right\rangle\quad \forall\,\Phi={}^{t}(\varphi,\psi,\phi)\in\mathcal{D}(\mathcal{A})
$$
or equivalently,
\begin{equation}\label{WPTE25}
\left\langle(-\mathcal{A}+J)\Phi,\left(U_{*}-\widetilde{U}_{*}\right)\right\rangle=0\quad \forall\,\Phi={}^{t}(\varphi,\psi,\phi)\in\mathcal{D}(\mathcal{A}).
\end{equation}

Now we need to prove that $(-\mathcal{A}+J)$ is surjective. For this aim let's consider ${}^{t}(f,g,h)\in\mathcal{H}$ and we have to look for $(u,v,\theta)\in\mathcal{D}(\mathcal{A})$ such that 
$$
(-\mathcal{A}+J)\left(\begin{array}{c}
u
\\
v
\\
\theta
\end{array}\right)=\left(\begin{array}{c}
f
\\
g
\\
h
\end{array}\right)
$$
which can be recast as follow
$$
\left\{\begin{array}{l}
v=-u-f
\\
-u''+\kappa\theta'-v=g
\\
-(a\theta')'+\kappa v'+\theta=h.
\end{array}\right.
$$
Except the first line this correspond to the same problem as \eqref{WPTE26} and it turns out that $(-\mathcal{A}+J)$ is surjective. 

By coming back to \eqref{WPTE25} this allow us to choose $\Phi\in\mathcal{D}(\mathcal{A})$ such that $(-\mathcal{A}+J)\Phi=\left(U_{*}-\widetilde{U}_{*}\right)$ then we obtain that $\|\widetilde{U}-\widetilde{U}_{*}\|_{\mathcal{H}}=0$ which implies that $\widetilde{U}_{*}=U_{*}$. Thus $\widetilde{U}_{*}\in\mathcal{D}(\mathcal{A})$ and this shows that $\mathcal{D}(\mathcal{A}^{*})\subset\mathcal{D}(\mathcal{A})$ and $\mathcal{A}^{*}$ is as given in \eqref{WPTE23}. This conclude the proof.
\end{proof}
Now we have proved that $\mathcal{A}$ is densely defined, dissipative and closed operator. Besides, from Proposition \ref{WPTE2} and Proposition \ref{WPTE3} for every $U=(u,v,\theta)\in\mathcal{D}(\mathcal{A})=\mathcal{D}(\mathcal{A}^{*})$ we have
$$
\mathrm{Re}\left\langle\mathcal{A}^{*}U,U\right\rangle_{H}=\mathrm{Re}\left\langle U,\mathcal{A}U\right\rangle_{H}=-\int_{0}^{1}a(s)|\theta'(s)|^{2}\,\ud s
$$
this means that $\mathcal{A}^{*}$ is a dissipative operator as well. Therefore, according to \cite[Proposition 3.1.11]{TW} $\mathcal{A}$ is m-dissipative and by the Lumer-Phillips theorem \cite[Theorem 4.3]{Pazy} we have the following:
\begin{thm}
The operator $\mathcal{A}$ is a generator of a $C_{0}$-semigroup of contraction in the space $\mathcal{H}$. 
\end{thm}
\section{Stability result}\label{PSTE}
\subsection{Strong stability}
In this section we shall prove that the semigroup associated to the operator $\mathcal{A}$ is strongly stable. For this aim we use the criteria of Arendt-Batty \cite{Are-Bat}, following which the $C_{0}$-semigroup of contractions $e^{tA}$ in a Banach space is strongly stable, if $A$ has no pure imaginary eigenvalues and $\sigma(A)\cap i\R$ contains only a countable number of elements. In particular, in our case we will show that $\sigma(A)\cap i\R$ is an empty set. 
\begin{prop}
For all $\lambda\in\R$, the operator $i\lambda I-\mathcal{A}$ is injective.
\end{prop}
\begin{proof}
Let $\lambda\in\R$ and let's consider the following eigenvalue problem 
\begin{equation}\label{PSTE32}
\mathcal{A}(u,v,\theta)=i\lambda(u,v,\theta)
\end{equation}
with $(u,v,\theta)\in\mathcal{D}(\mathcal{A})$ which written as follows
\begin{equation}\label{PSTE26}
\left\{\begin{array}{l}
v=i\lambda u
\\
u''-\kappa\theta'=i\lambda v
\\
(a\theta')'-\kappa v'=i\lambda\theta.
\end{array}\right.
\end{equation}
Using the dissipation property of the operator $\mathcal{A}$ then by taking the real part of the inner product in $\mathcal{H}$ of \eqref{PSTE32} with $(u,v,\theta)$, we have
$$
\re\langle\mathcal{A}(u,v,\theta),(u,v,\theta)\rangle_{\mathcal{H}}=-\int_{0}^{1}a(x)|\theta'(x)|^{2}\,\ud x= 0.
$$
Since $a>0$ in $(0,1)$ then $\theta'=0$ a.e. in $(0,1)$ and due to the boundary conditions we have $\theta=0$ a.e. in $(0,1)$. Putting this in equation \eqref{PSTE26} we find that $v'=0$ a.e. in $(0,1)$ then again by the boundary conditions we follow that $v=0$ a.e. in $(0,1)$. This together with \eqref{PSTE26} gives that $u''=0$ in $(0,1)$, which by the boundary conditions leads to $u=0$ in $(0,1)$. So that, the operator $i\lambda I-\mathcal{A}$ is injective.
\end{proof}
At this point we set the function
\begin{equation*}
W(x)=\int_{0}^{x}\frac{\ud s}{a(s)}\;\;\text{ if }\, \frac{1}{a}\in L^{1}(0,1) \;\quad\text{ or }\;\quad W(x)=\int_{x}^{1}\frac{\ud s}{a(s)}\;\;\text{ if }\; \frac{1}{a}\notin L^{1}(0,1).
\end{equation*}
We recall the following technical lemma of Young inequality for the integral operators.
\begin{prop}
If $W\in L^{1}(0,1)$, the operator $\mathcal{A}$ is invertible and $\mathcal{A}^{-1}$ is compact.
\end{prop}
\begin{proof}
Let $(f,g,h)\in\mathcal{H}$ and we are looking for $(u,v,\theta)\in\mathcal{D}(\mathcal{A})$ solution of the following problem
\begin{equation*}
\mathcal{A}\left(\begin{array}{c}
u
\\
v
\\
\theta
\end{array}\right)=\left(\begin{array}{c}
f
\\
g
\\
h
\end{array}\right),
\end{equation*}
which can be recast as follow
\begin{equation}\label{PSTE14}
\left\{\begin{array}{l}
v=f
\\
u''-\kappa\theta'=g
\\
(a\theta')'-\kappa v'=h.
\end{array}\right.
\end{equation}
This implies that $v$ is well-defined and we need to do so for $u$ and $\theta$. By substituting $v$ in the third line of \eqref{PSTE14} and integrating the second line two times, we obtain
\begin{equation}\label{PSTE15}
\left\{\begin{array}{l}
\ds u(x)=\kappa\int_{0}^{x}\theta(s)\,\ud s+\int_{0}^{x}\int_{0}^{s}g(\tau)\,\ud\tau\,\ud s+C_{1}x
\\
(a\theta')'(x)=\kappa f'(x)+h(x),
\end{array}\right.
\end{equation}
where the constant $C_{1}$ is defined as follows in order to guaranty that $u(1)=0$, 
\begin{equation}\label{PSTE20}
C_{1}=-\kappa\int_{0}^{1}\theta(s)\,\ud s-\int_{0}^{1}\int_{0}^{s}g(\tau)\,\ud\tau\,\ud s.
\end{equation}
It follows that $u$ is well-defined and belongs to $H_{0}^{1}(0,1)$ as long as $\theta$ exists and belongs to $L^{2}(0,1)$ and we have
\begin{equation}\label{PSTE22}
u'(x)=\kappa\theta(x)+\int_{0}^{x}g(s)\,\ud s+C_{1}.
\end{equation}
The rest of the proof is devoted to find $\theta$ to prove that the triplet $(u,v,\theta)$ belongs to $\mathcal{D}(\mathcal{A})$ and $\mathcal{A}^{-1}$ is a compact operator in $\mathcal{H}$. For this aim we will discuss separately the two case of the weak and strong degeneracy.

\medskip

\textbf{(WDP):} Integration two times the second equation of \eqref{PSTE15} then we obtain
\begin{align}
\theta(x)&=\int_{0}^{x}\frac{1}{a(s)}\int_{0}^{s}(\kappa f'+h)(\tau)\,\ud\tau\,\ud s+C_{2}W(x)\label{PSTE16}
\\
&=W(x)\int_{0}^{x}(\kappa f'+h)(s)\,\ud s-\int_{0}^{x}W(s)(\kappa f'+h)(s)\,\ud s+C_{2}W(x),\label{PSTE17}
\end{align}
with $C_{2}$ is set in order to guaranty that $\theta(1)=0$ and is given by
\begin{align}\label{PSTE18}
C_{2}&=-\frac{1}{W(1)}\int_{0}^{1}\frac{1}{a(s)}\int_{0}^{s}(\kappa f'+h)(\tau)\,\ud\tau\,\ud s\nonumber
\\
&=\int_{0}^{1}\frac{W(s)-W(1)}{W(1)}(\kappa f'+h)(s)\,\ud s,
\end{align}
where an integration by parts have been performed in \eqref{PSTE17} and \eqref{PSTE18}. This allow us to write $\theta$ in the following form
\begin{equation}\label{PSTE19}
\theta(x)=\int_{0}^{1}K(x,s)(\kappa f'+h)(s)\,\ud s,
\end{equation}
where
$$
K(x,s)=\left\{\begin{array}{ll}
\frac{W(s)(W(x)-W(1))}{W(1)}&\text{if } 0\leq s<x\leq 1
\\
\frac{W(s)}{W(1)}-1&\text{if } 0\leq x\leq s\leq 1.
\end{array}\right.
$$

Now we can observe from \eqref{PSTE16} that $\theta\in L^{2}(0,1)$, absolutely continuous in $[0,1]$, and satisfying $a\theta'\in H^{1}(0,1)$ since 
\begin{align*}
a\theta'=\int_{0}^{x}(\kappa f'+h)(s)\,\ud s \quad\text{ and }\quad (a\theta')'=(\kappa f'+h)(s),
\end{align*}
where clearly both of them belong to $L^{2}(0,1)$, also since $\frac{1}{\sqrt{a}}\in L^{2}(0,1)$ then we can verify that $\sqrt{a}\theta'\in L^{2}(0,1)$ and obviously we have $\theta(0)=\theta(1)=0$. All these prove that the triplet $(u,v,\theta)\in\mathcal{D}(\mathcal{A})$ defined in \eqref{PSTE14}, \eqref{PSTE15} and \eqref{PSTE16} and as a result $\mathcal{A}$ is invertible and from \eqref{PSTE14}, \eqref{PSTE15} and \eqref{PSTE19} we have
$$
\mathcal{A}^{-1}\left(\begin{array}{c}
f
\\
g
\\
h
\end{array}\right)=\left(\begin{array}{c}
u
\\
v
\\
\theta
\end{array}\right),
$$
where
\begin{align}\label{PSTE24}
\left\{\begin{array}{l}
\ds u(x)=\kappa\int_{0}^{x}\int_{0}^{t}\frac{1}{a(s)}\int_{0}^{s}(\kappa f'+h)(\tau)\,\ud\tau\,\ud s\,\ud t+C_{2}\int_{0}^{x}W(s)\,\ud s+\int_{0}^{x}\int_{0}^{s}g(\tau)\,\ud\tau\,\ud s+C_{1}x
\\
v(x)=f(x)
\\
\ds \theta(x)=\int_{0}^{1}K(x,s)(\kappa f'+h)(s)\,\ud s.
\end{array}\right.
\end{align}
Now it's still to show that $\mathcal{A}^{-1}$ is a compact operator (in particular bounded) in $\mathcal{H}$. For this we set
$$
\widetilde{W}(x)=\int_{0}^{x}W(s)\,\ud s,
$$
which is well-defined for all $x\in[0,1]$ since $W\in L^{2}(0,1)$ and we recall the expression of $C_{1}$ in \eqref{PSTE20} and performing the following calculation based on some integration by parts after substituting \eqref{PSTE17} into \eqref{PSTE20}, we follow
\begin{align}\label{PSTE21}
C_{1}&=\kappa\int_{0}^{1}(1-s)W(s)(kf'+h)(s)\,\ud s-\kappa\int_{0}^{1}W(s)\int_{0}^{s}(\kappa f'+h)(\tau)\,\ud\tau\,\ud s\nonumber
\\
&-\kappa\widetilde{W}(1)C_{2}+\int_{0}^{1}(s-1)g(s)\,\ud s\nonumber
\\
&=\kappa\int_{0}^{1}(1-s)W(s)(kf'+h)(s)\,\ud s+
\kappa\int_{0}^{1}\widetilde{W}(s)(\kappa f'+h)(s)\,\ud s
\nonumber
\\
&
-\kappa\widetilde{W}(1)\int_{0}^{1}(\kappa f'+h)(s)\,\ud s
-\kappa\widetilde{W}(1)C_{2}+\int_{0}^{1}(s-1)g(s)\,\ud s\nonumber
\\
&=
\int_{0}^{1}
\kappa\left(
\left(1-\frac{\widetilde{W}(1)}{W(1)}-s\right)W(s)
+\widetilde{W}(s)
\right)
(\kappa f'+h)(s)\,\ud s+
\int_{0}^{1}(s-1)g(s)\,\ud s,
\end{align}
where we have using the expression of $C_{2}$ in \eqref{PSTE18}. 
\medskip

Now by substituting \eqref{PSTE19} and \eqref{PSTE21} into \eqref{PSTE22} we can write $u'$ in the following form
\begin{align}\label{PSTE23}
u'(x)&=
\int_{0}^{1}\kappa\left(\left(1-\frac{\widetilde{W}(1)}{W(1)}-s\right)W(s)
+\widetilde{W}(s)+K(x,s)\right)(\kappa f'+h)(s)\,\ud s\nonumber
\\
&+\int_{0}^{1}(s-1)g(s)\,\ud s+\int_{0}^{x}g(s)\,\ud s\nonumber
\\
&=\int_{0}^{1}M(x,s)(\kappa f'+h)(s)\,\ud s+\int_{0}^{1}N(x,s)g(s)\,\ud s,
\end{align}
where
$$
M(x,s)=\kappa\left(\left(1-\frac{\widetilde{W}(1)}{W(1)}-s\right)W(s)
+\widetilde{W}(s)+K(x,s)\right)
$$
and
$$
N(x,s)=\left\{\begin{array}{ll}
s&\text{if }0\leq s<x\leq 1
\\
s-1&\text{if }0\leq x\leq s\leq 1.
\end{array}\right.
$$

We set
$$
Az=\int_{0}^{1}M(x,s)z(s)\,\ud s,\quad Bz=\int_{0}^{1}N(x,s)z(s)\,\ud s\text{ and }Cz=\int_{0}^{1}K(x,s)z(s)\,\ud s
$$
which are compact operators in $L^{2}(0,1)$ as a Hilbert-Schmidt operators since the kernels $M$, $N$ and $K$ belong to $L^{2}((0,1)^{2})$ as $\ds\frac{1}{a}$ belongs to $L^{1}(0,1)$ and $W$ is an increasing function and belongs to $L^{1}(0,1)$. 

Now let's consider a bounded sequence $(f_{n},g_{n},h_{n})\in\mathcal{H}$ and we set
$$
\left(\begin{array}{c}
u_{n}
\\
v_{n}
\\
\theta_{n}
\end{array}\right)=\mathcal{A}^{-1}\left(\begin{array}{c}
f_{n}
\\
g_{n}
\\
h_{n}
\end{array}\right)
$$
and we shall prove that there exists a sub-sequence $(u_{\varphi(n)},v_{\varphi(n)},\theta_{\varphi(n)})$ that converge in $\mathcal{H}$. In another words we need to show that there exists $(\tilde{u},\tilde{v},\tilde{\theta})\in\mathcal{H}$ such that $u_{\varphi(n)}'$, $v_{\varphi(n)}$ and $\theta_{\varphi(n)}$ converge respectively to $\tilde{u}'$, $\tilde{v}$ and $\theta$ in $L^{2}(0,1)$. Following to \eqref{PSTE23} and \eqref{PSTE24} and the above notations we can write
$$
\left\{\begin{array}{l}
u_{n}'=\kappa\,A\,f_{n}'+A\,h_{n}+B\,g_{n}
\\
v_{n}=f_{n}
\\
\theta_{n}= \kappa\,C\,f_{n}'+C\,h_{n}.
\end{array}\right.
$$
First, observe that the sequence $(f_{n})$ is bounded in $H^{1}_{0}(0,1)$ then the compactness of the injection $H_{0}^{1}(0,1)\hookrightarrow L^{2}(0,1)$ implies there exists a sub-sequence $(v_{\varphi(n)})$ that converges in $L^{2}(0,1)$ to $\tilde{v}$. Next, as $A$, $B$ and $C$ are compact operators in $L^{2}(0,1)$ and the sequences $(f_{\varphi(n)}')$, $(g_{\varphi(n)})$ and $(h_{\varphi(n)})$ are bounded in $L^{2}(0,1)$ then we extract a sub-sequence (still noted the same) such that $u_{\varphi(n)}'=\kappa\,A\,f_{n}'+A\,h_{n}+B\,g_{n}$ converges to $z$ in $L^{2}(0,1)$ and $\theta_{\varphi(n)}= \kappa\,C\,f_{\varphi(n)}'+C\,h_{\varphi(n)}$ converge as well in $L^{2}(0,1)$ to $\tilde{\theta}$. Finally, the sequence $(u_{\varphi(n)})$ belongs to $H_{0}^{1}(0,1)$ and $(u_{\varphi(n)}')$ is a Cauchy sequence in $L^{2}(0,1)$ which implies that $(u_{\varphi(n)})$ a Cauchy sequence in $H_{0}^{1}(0,1)$, therefore $(u_{\varphi(n)})$ converge to some $\tilde{u}$ in $H_{0}^{1}(0,1)$. And this complete the proof of this part and prove that $\mathcal{A}^{-1}$ is compact.

\medskip

\textbf{(SDP):} As done in the previous case we can solve the second equation of \eqref{PSTE15} by taking into account the boundary conditions and proceeding by integration by parts then
\begin{align}
\theta(x)&=-\int_{x}^{1}\frac{1}{a(s)}\int_{0}^{s}(\kappa f'+h)(\tau)\,\ud\tau\,\ud s\label{PSTE25}
\\
&=-W(x)\int_{0}^{x}(\kappa f'+h)(s)\,\ud s-\int_{x}^{1}W(s)(\kappa f'+h)(s)\,\ud s\nonumber
\\
&=\int_{0}^{1}K(x,s)(\kappa f'+h)(s)\,\ud s\label{PSTE28}
\end{align}
where
$$
K(x,s)=\left\{\begin{array}{ll}
-W(x)&\text{if }0\leq s<x\leq 1
\\
-W(s)&\text{if }0\leq x\leq s\leq 1.
\end{array}\right.
$$
Observing that for all $x\in(0,1]$
$$
\int_{0}^{x}W(s)\,\ud s=\int_{0}^{x}\int_{s}^{1}\frac{\ud\tau}{a(\tau)}\,\ud s\geq \int_{0}^{x}\int_{x}^{1}\frac{\ud\tau}{a(\tau)}\,\ud s=xW(x),
$$
which proves that the function $x\mapsto xW(x)$ is bounded in $[0,1]$ and converge to $0$ as $x$ goes to $0$. This allow us to write after an integration by parts that
\begin{equation}\label{PSTE29}
\int_{0}^{1}W(x)\,\ud x=\int_{0}^{1}\int_{x}^{1}\frac{\ud s}{a(s)}\,\ud x=\int_{0}^{1}\frac{x}{a(x)}\,\ud x.
\end{equation}
\\
Moreover we have also
\begin{align*}
\int_{0}^{1}\int_{0}^{1}|K(x,s)|^{2}\,\ud s\ud x&=\int_{0}^{1}\int_{0}^{x}|K(x,s)|^{2}\,\ud s\,\ud x+\int_{0}^{1}\int_{x}^{1}|K(x,s)|^{2}\,\ud s\,\ud x
\\
&=\int_{0}^{1}x|W(x)|^{2}\,\ud x+\int_{0}^{1}\int_{x}^{1}|W(s)|^{2}\,\ud s\,\ud x
\\
&=\int_{0}^{1}x|W(x)|^{2}\,\ud x+\int_{0}^{1}\int_{0}^{s}|W(s)|^{2}\,\ud x\,\ud s
\\
&=2\int_{0}^{1}x|W(x)|^{2}\,\ud x<+\infty,
\end{align*}
as the function $x\mapsto xW(x)$ is bounded in $[0,1]$ and $W\in L^{1}(0,1)$ which implies that $K$ belongs to $L^{2}((0,1)\times(0,1))$. Therefore, we conclude that
$$
f\mapsto\int_{0}^{1}K(x,s)f(s)\,\ud s
$$
is a compact operator in $L^{2}(0,1)$ as a Hilbert-Schmidt operator.

It is clear that from \eqref{PSTE25} that $\theta$ is locally absolutely continuous in $(0,1]$, $a\theta'=\int_{0}^{x}(\kappa f'+h)(s)\,\ud s$ belong to $H^{1}(0,1)$, $(a\theta')(0)=\theta(1)=0$ and $\ds \sqrt{a}\theta'\in L^{2}(0,1)$in fact, by H\"older inequality and \eqref{PSTE29} we have
\begin{align*}
\int_{0}^{1}a(x)|\theta'(x)|^{2}\,\ud x&=\int_{0}^{1}\frac{1}{a(x)}\left|\int_{0}^{x}(\kappa f'+h)(s)\,\ud s\right|^{2}\,\ud x
\\
&\leq\int_{0}^{1}\frac{x}{a(x)}\int_{0}^{x}(\kappa f'+h)^{2}(s)\ud s\,\ud x
\\
&\leq\int_{0}^{1}\frac{x}{a(x)}\,\ud x\int_{0}^{1}(\kappa f'+h)^{2}(x)\,\ud x
\\
&\leq \int_{0}^{1}W(x)\,\ud x\int_{0}^{1}(\kappa f'+h)^{2}(x)\,\ud x<+\infty.
\end{align*}
This prove that $(u,v,\theta)\in\mathcal{D}(\mathcal{A})$ and consequently $\mathcal{A}$ is invertible and thanks to \eqref{PSTE14}, \eqref{PSTE15}, \eqref{PSTE25} and \eqref{PSTE28}, one gets
$$
\mathcal{A}^{-1}\left(\begin{array}{c}
f
\\
g
\\
h
\end{array}\right)=\left(\begin{array}{c}
u
\\
v
\\
\theta
\end{array}\right)
$$
where
$$
\left\{\begin{array}{l}
\ds u(x)=-\kappa\int_{0}^{x}\int_{t}^{1}\frac{1}{a(s)}\int_{0}^{s}(\kappa f'+h)(\tau)\,\ud\tau\,\ud s\,\ud t+\int_{0}^{x}\int_{0}^{s}g(\tau)\,\ud\tau\,\ud s+C_{1}x
\\
v(x)=f(x)
\\
\ds\theta(x)=\int_{0}^{1}K(x,s)(\kappa f'+h)(s)\,\ud s.
\end{array}\right.
$$
Now we need to prove that $\mathcal{A}^{-1}$ is a compact operator. We shall proceed the same way as the previous case. So now we have only to write $u'$ in same form as \eqref{PSTE23}, if so the rest of the proof is exactly the same done in the case of the week damped problem. First, we want to recall that
\begin{equation}\label{PSTE30}
u'(x)=\kappa\int_{0}^{1}K(x,s)(\kappa f'+h)(s)\,\ud s+\int_{0}^{x}g(s)\,\ud s+C_{1},
\end{equation}
where by substituting \eqref{PSTE25} into \eqref{PSTE20} and integrating by parts
\begin{align}\label{PSTE31}
C_{1}&=\kappa\int_{0}^{1}\int_{t}^{1}\frac{1}{a(s)}\int_{0}^{s}(\kappa f'+h)(\tau)\,\ud\tau\,\ud s\,\ud t-\int_{0}^{1}\int_{0}^{s}g(\tau)\,\ud\tau\,\ud s\nonumber
\\
&=\kappa\int_{0}^{1}\frac{s}{a(s)}\int_{0}^{s}(\kappa f'+h)(\tau)\,\ud\tau\,\ud s+\int_{0}^{1}(s-1)g(s)\,\ud s\nonumber
\\
&=\kappa\int_{0}^{1}\int_{s}^{1}\frac{\tau}{a(\tau)}\,\ud\tau(\kappa f'+h)(s)\,\ud s+\int_{0}^{1}(s-1)g(s)\,\ud s.
\end{align}
Substituting \eqref{PSTE31} into \eqref{PSTE30}, we derive by setting
$$
\widetilde{W}(x)=\int_{x}^{1}\frac{s}{a(s)}\,\ud s,
$$
that
$$
u'(x)=\int_{0}^{1}M(x,s)(\kappa f'+h)(s)\,\ud s+\int_{0}^{1}N(x,s)g(s)\,\ud s
$$
where
$$
M(x,s)=\kappa K(x,s)+\kappa\widetilde{W}(s)
$$
and $N$ is the same as above (of the first case). Now noting that all kernels $K$, $M$ and $N$ belong to $L^{2}((0,1)^{2})$ the integral operators defining $u'$ are compact operators in $L^{2}(0,1)$. With this, all ingredients are proven that allow us to prove in the same way the compactness property of the inverse of $\mathcal{A}$. This conclude the prove the second case and conclude the proof. 
\end{proof}
\begin{thm}
We assume that $W\in L^{1}(0,1)$ then the imaginary axis is a subset of the resolvent set of $\mathcal{A}$, i.e. $i\R\subset\rho(\mathcal{A})$. And the semigroup $e^{t\mathcal{A}}$ is strongly stable in $\mathcal{H}$ which means that
$$
\lim_{t\to+\infty}\|e^{t\mathcal{A}}(u_{0},v_{0},\theta_{0})\|_{\mathcal{H}}=0\quad\forall\,(u_{0},v_{0},\theta_{0}) \in \mathcal{D}(\mathcal{A}).
$$ 
\end{thm}
\begin{proof}
Since $\mathcal{A}$ have no eigenvalue in the imaginary axis and the inverse of $\mathcal{A}$ is compact then we conclude that $i\R$ belongs to the resolvent set $\rho(\mathcal{A})$ which implies that in particular the strong stability of the semigroup.
\end{proof}
\subsection{Uniform decay rate}
This section is devoted to prove the exponential stability of system \eqref{ITE1}. This proof is based in some resolvent estimate namely we have the following 
\begin{prop}\cite{huang,pruss}\label{PSTE1}
Let $e^{t\mathcal{B}}$ be a bounded $C_{0}$-semigroup on a Hilbert space $X$ with generator $\mathcal{B}$ such that $i\R\subset\rho(\mathcal{B})$. Then $e^{t\mathcal{B}}$ is exponentially stable i.e. there exist $C>0$ and $\omega>0$ such that
$$
\|e^{t\mathcal{B}}\|_{\mathcal{L}(X)}\leq C\,e^{-\omega t}\quad\forall\, t\geq 0,
$$
if and only if 
$$
\sup \left\{\|(i\lambda I-\mathcal{B})^{-1}\|_{\mathcal{L}(X)}, \lambda \in \mathbb{R}\right\}<\infty.
$$
\end{prop}
Our main result is given by the following
\begin{thm}
	\label{th: stabilization}
We assume $a(x)$ satisfied (WDP) or (SDP) such that $W\in L^{1}(0,1)$. Then 
\begin{equation}\label{est: spect 1}
\limsup_{\lambda \in \mathbb{R}, |\lambda|\rightarrow\infty}\|(i\lambda I-\mathcal{A})^{-1}\|_{\mathcal{L}(\mathcal{H})}<\infty.
\end{equation}
And there exist $C>0$ and $\omega>0$ such that
$$
\| e^{t\mathcal{A}}U\|_{\mathcal{H}}\le C e^{-\omega t} \, \|U\|_{\mathcal{H}}\quad \forall\,U\in\mathcal{H}.
$$  
\end{thm}


%
%

\medskip

We prove \eqref{est: spect 1} by  contradiction. As $\mathcal{A}-i\lambda $ is invertible, we only to prove the estimation for 
large $|\lambda|$. If \eqref{est: spect 1} is false for every $C>0$, there exist sequences $(f_n, g_n,\theta_n)$, $(u_n,v_n,\theta_n)$ and $\lambda_n$ such that
$|\lambda_n|\to \infty$, 
\begin{equation}
	\label{syst: spect 1}
\begin{cases}
v_{n}-i\lambda_{n}u_{n}=f_{n},
\\
u_{n}''-\kappa\theta_{n}'-i\lambda_{n}v_{n}=g_{n}, 
\\
(a\theta_{n}')'-\kappa v_{n}'-i\lambda_{n}\theta_{n}=h_{n},
\end{cases}
\end{equation}
where $\| f_n'\|+\|g_n\|+\|h_n\|\to 0$ 
as $n\to 0$, and $ \|u'_n\|+\| v_n\|+\|\theta_n\|=1$ for all $n\in\N$ by normalization with $\|\,.\,\|=\|\,.\,\|_{L^2(0,1)}$. 
As $v_n=f_n+i\lambda_{n}u_{n}$, from \eqref{syst: spect 1} we have
\begin{equation}
	\label{syst: spect 2}
\begin{cases}
u_{n}''-\kappa\theta_{n}'+\lambda_{n}^2u_{n}=g_{n}+i\lambda_nf_n=F_n, 
\\
(a\theta_{n}')'-i\kappa \lambda_{n}u_{n}'-i\lambda_{n}\theta_{n}=h_{n}+\kappa f'_n=G_n.
\end{cases}
\end{equation}
%
%
\begin{lem}
	\label{lem: est. spec. 1}
Let $\theta_n$ and $\lambda_n$ satisfying \eqref{syst: spect 2}, we have
\[
\int_0^1a(x)|\theta'_n(x)|^2\,\ud x\to 0 \text{ as }n\to\infty.
\]
\end{lem}
\begin{proof}
We denote $\pd{\,.\,,\,.\,}_{L^{2}(0,1)}=\pd{\,.\,,\,.\,}$.
We take the inner product of the first equation of \eqref{syst: spect 2} with $u_n$ and the second equation with $\theta_n$, we obtain
after integration by parts
\begin{equation}
\begin{cases}
-\|u_{n}'\|^2-\kappa\pd{\theta_{n}',u_n}+\lambda_{n}^2\|u_{n}\|^2=\pd{F_n,u_n} \\
-\pd{a\theta_{n}',\theta_n'}-i\kappa \lambda_{n}\pd{u_{n}',\theta_n}-i\lambda_{n}\|\theta_{n}\|^2=\pd{G_n,\theta_n}.
\end{cases}
\end{equation}
We multiply the first equation by $-i\lambda_n$ and we add with the second, we obtain
\begin{align}
	\label{eq:  spect 3}
&i\lambda_n\|u_{n}'\|^2+i\lambda_n\kappa\pd{\theta_{n}',u_n}-i\lambda_{n}^3\|u_{n}\|^2-\pd{a\theta_{n}',\theta_n'}
-i\kappa \lambda_{n}\pd{u_{n}',\theta_n}-i\lambda_{n}\|\theta_{n}\|^2\\
&\quad\quad=-i\lambda_n\pd{F_n,u_n}+\pd{G_n,\theta_n}. \notag
\end{align}
As $\pd{\theta_{n}',u_n}-\pd{u_{n}',\theta_n}=\pd{\theta_{n}',u_n}+\pd{u_{n},\theta_n'}$ is real, taking the real part of \eqref{eq:  spect 3}
we obtain the result as $|\lambda_n|\|u_n\|$ and $\|\theta_n\|$ are bounded.
\end{proof}
%
%
\begin{lem}
	\label{lem: est. spec. 2}
Let $\theta_n$ and $\lambda_n$ satisfying \eqref{syst: spect 2}, then
\[
\int_0^1|\theta_n(x)|^2\,\ud x\to 0 \text{ as }n\to\infty.
\]	
\end{lem}
\begin{proof}
First we have
\[
|\theta_n(x)|^2\le \left| \int_x^1\theta'_n(s)\,\ud s  \right|^2\le  \int_x^1a(s)|\theta'_n(s)|^2\,\ud s \int_x^1\frac{\ud s}{a(s)}
\le  C W(x)
\int_x^1a(s)|\theta'_n(s)|^2\,\ud s .
\]
Integrating this estimate in $x\in (0,1)$ and since $W$ belongs to $L^{1}(0,1)$, one gets
\[
\int_0^1|\theta_n(x)|^2\,\ud x
\le  C\int_0^1W(x)\int_{x}^{1}a(s)|\theta'_n(s)|^2\,\ud s\,\ud x\le C\int_0^1 a(s)|\theta'_n(s)|^2\,\ud s.
\]
We obtain the result from Lemma~\ref{lem: est. spec. 1}.
\end{proof}
%
%
\begin{lem}
	\label{lem: est. traces}
Let $(u_n)$ and $(\theta_n)$ be the solution of \eqref{syst: spect 2}, we have
\begin{align*}
&\lambda_n^{-1/2}|u'_n(0)-\kappa\theta_n(0)|\to 0 \text{ as }n\to\infty, \\
&\lambda_n^{-1/2} |u'_n(1)| \to 0 \text{ as }n\to\infty ,\\
&\lambda_n^{-1/2}|(a\theta_n)'(0)|+\lambda_n^{-1/2}|(a\theta_n)'(1)|\to 0 \text{ as }n\to\infty  . 
\end{align*}
\end{lem}
\begin{proof}
%
%

We prove the first item, the proof of the second is analogous, using $\theta_n(1)=0$.
From \eqref{syst: spect 2},  $w_n= u_n'-\kappa\theta_n$ and $u_n$ satisfy the system
\begin{equation*}
\begin{cases}
&w_n'+\lambda_n^2 u_n=F_n\\
&u_n'-w_n=\kappa\theta_n.
\end{cases}
\end{equation*}
Let $\varphi\in \mathcal{C}^\infty$ be such that $\varphi(x) =1 $ in a neighborhood of $x=0$ and supported in $x\le 1/2$.
We have 
\begin{align}
	\label{est: trace w}
2\re \pd{ w_n'+\lambda_n^2 u_n ,\varphi w_n}+2\re\pd{u_n'-w_n,\varphi \lambda_n^2 u_n}=2\re\pd{F_n, \varphi w_n}
+2\re\pd{\kappa\theta_n , \varphi \lambda_n^2 u_n }.
\end{align}
As $u_n(0)=0$ we have

\[
2\re \pd{ w_n' ,\varphi w_n}=-|w_n(0)|^2- \pd{\varphi' w_n,w_n} \text{ and }
2\re \pd{ u_n' ,\varphi \lambda_n^2 u_n}=- \lambda_{n}^2 \pd{\varphi' u_n,u_n}.
\]

From \eqref{est: trace w} we deduce
\[
|w_n(0)|^2\le C (  \|w_n\|^2+\lambda_{n}^{2}\|u_{n}\|^{2}+\|F_n\| \|w_n\|+\lambda_n \|\theta_n\| \| \lambda_n u_n\|).
\]
This implies the first item of the lemma, as $  \|w_n\|$ and  $\| \lambda_n u_n\|$ are bounded and $\|\theta_{n}\|$  and 
$\lambda_n^{-1}\|F_n\| $  go to $0$ as $n\to \infty$.

\medskip

For the last property we compute from \eqref{syst: spect 2}
\begin{equation}
2\re \pd{(a\theta_{n}')'-i\kappa \lambda_{n}u_{n}'-i\lambda_{n}\theta_{n}, \varphi a\theta_ {n}'}=2\re\pd{G_n,  \varphi a\theta_ {n}' }.
\end{equation}
We have 
\[
2\re \pd{(a\theta_{n}')', \varphi a\theta_n'}=-|(a\theta_n')(0)|^2-\pd{\varphi' a\theta_n',  a\theta'_n}.
\]
This implies 
\[
|(a\theta_ {n}')(0)|^2\le C(\lambda_n\|u_n'\|\|  a\theta_ {n}' \|+ \lambda_{n}\|\theta_{n}\| \|  a\theta_ {n}' \|+ \|G_n\|  \|  a\theta_ {n}' \|),
\]
this gives the result as $\|u_n'\|$ is bounded and, $\|  a\theta_ {n}' \|$,  $\|\theta_{n}\|$ and  $ \|G_n\| $ goes to 0 as $n\to \infty$.
\end{proof}
Now from the second line of \eqref{syst: spect 2} we have
\begin{equation}\label{PSTE33}
u_n'=\frac{i}{\kappa\lambda_{n}}G_{n}-\frac{1}{\kappa}\theta_{n}-\frac{i}{\kappa\lambda_{n}}(a\theta_n')'.
\end{equation}

Multiplying \eqref{PSTE33} by $\overline{u'}$ and integrating over $(0,1)$ then an integration by parts gives
\begin{align*}
\int_{0}^{1}|u_{n}'|^{2}\,\ud x
&=\frac{i}{\kappa\lambda_{n}}\int_{0}^{1}\! G_{n}\overline{u'}_{n}\,\ud x-\frac{1}{\kappa}\int_{0}^{1}\!\theta_{n}\overline{u'}_{n}\,\ud x-\frac{i}{\kappa\lambda_{n}}\int_{0}^{1}\!(a\theta_{n}')'(\overline{u_{n}'-\kappa\theta_{n}})\,\ud x-\frac{i}{\lambda_{n}}\int_{0}^{1}\!(a\theta_{n}')\overline{\theta_{n}}\,\ud x
\\
&=\frac{i}{\kappa\lambda_{n}}\int_{0}^{1}G_{n}\overline{u'}_{n}\,\ud x-\frac{1}{\kappa}\int_{0}^{1}\theta_{n}\overline{u'}_{n}\,\ud x+\frac{i}{\kappa\lambda_{n}}\int_{0}^{1}(a\theta_{n}')(\overline{u_{n}'-\kappa\theta_{n}})'\,\ud x+\frac{i}{\lambda_{n}}\int_{0}^{1}a|\theta_{n}'|^{2}\,\ud x
\\
&+\frac{i}{\kappa\lambda_{n}}\left((a\theta_{n}')(0)(\overline{u_{n}'-\kappa\theta_{n}})(0)-(a\theta_{n}')(1)(\overline{u_{n}'-\kappa\theta_{n}})(1)\right)
\\
&=\frac{i}{\kappa\lambda_{n}}\int_{0}^{1}G_{n}\overline{u'}_{n}\,\ud x-\frac{1}{\kappa}\int_{0}^{1}\theta_{n}\overline{u'}_{n}\,\ud x-\frac{i\lambda_{n}}{\kappa}\int_{0}^{1}a\theta_{n}'\overline{u_{n}}\,\ud x+\frac{i}{\kappa\lambda_{n}}\int_{0}^{1}a\theta_{n}'\overline{F_{n}}\,\ud x
\\
&+\frac{i}{\lambda_{n}}\int_{0}^{1}a|\theta_{n}'|^{2}\,\ud x+\frac{i}{\kappa\lambda_{n}}\left((a\theta_{n}')(0)(\overline{u_{n}'-\kappa\theta_{n}})(0)-(a\theta_{n}')(1)\overline{u_{n}'}(1)\right),	
\end{align*}
where we have used the first line of \eqref{syst: spect 2}. Therefore, 
\begin{align*}
\int_{0}^{1}|u_{n}'|^{2}\,\ud x&\leq C\bigg(\frac{1}{\lambda_{n}}\int_{0}^{1}|G_{n}|^{2}\,\ud x+\int_{0}^{1}|\theta_{n}|^{2}\,\ud x+\lambda_{n}\left(\int_{0}^{1}|u_{n}|^{2}\,\ud x\right)^{\frac{1}{2}}\left(\int_{0}^{1}a|\theta_{n}'|^{2}\,\ud x\right)^{\frac{1}{2}}
\\
&+\frac{1}{\lambda_{n}}\left(\int_{0}^{1}|F_{n}|^{2}\,\ud x\right)^{\frac{1}{2}}\left(\int_{0}^{1}a|\theta_{n}'|^{2}\,\ud x\right)^{\frac{1}{2}}+\frac{1}{\lambda_{n}}\int_{0}^{1}a|\theta_{n}'|^{2}\,\ud x
\\
&+\frac{1}{\lambda_{n}}\left(|(a\theta_{n}')(0)|.|(u_{n}'-\kappa\theta_{n})(0)|+|(a\theta_{n}')(1)|.|u_{n}'(1)|\right)\bigg).
\end{align*}
Following to Lemma \ref{lem: est. spec. 1}, Lemma \ref{lem: est. spec. 2} and Lemma \ref{lem: est. traces} we deduce that
\begin{equation}\label{PSTE34}
\int_{0}^{1}|u_{n}'|^{2}\,\ud x\to 0\text{ as }n\to \infty.
\end{equation}
Taking the inner product of the first line of \eqref{syst: spect 2} with $u_{n}$, then by integration by parts
\begin{align*}
\lambda_{n}^{2}\int_{0}^{1}|u_{n}|^{2}\,\ud x=\int_{0}^{1}|u_{n}'|^{2}\,\ud x-\kappa\int_{0}^{1}\theta_{n}\overline{u_{n}'}\,\ud x+\int_{0}^{1}F_{n}\overline{u_{n}}\,\ud x.
\end{align*}
By Young's inequality we get
$$
\lambda_{n}^{2}\int_{0}^{1}|u_{n}|^{2}\,\ud x\leq C\left(\int_{0}^{1}|u_{n}'|^{2}\,\ud x+\int_{0}^{1}|\theta_{n}|^{2}\,\ud x+
\frac{1}{\lambda_{n}^2}\int_{0}^{1}|F_{n}|^{2}\,\ud x\right).
$$
Hence, by Lemma \ref{lem: est. spec. 2} and \eqref{PSTE34} we find
$$
\lambda_{n}^{2}\int_{0}^{1}|u_{n}|^{2}\,\ud x\to 0\text{ as }n\to \infty,
$$
which implies from the first line of \eqref{syst: spect 1} that
\begin{equation}\label{PSTE35}
\int_{0}^{1}|v_{n}|^{2}\,\ud x\to 0\text{ as }n\to \infty.
\end{equation}
Finally, the combination of Lemma \ref{lem: est. spec. 2}, \eqref{PSTE34} and \eqref{PSTE35} leads to a contradiction with the assumption that $\|(u_{n},v_{n},\theta_{n})\|_{\mathcal{H}}=1$ for every $n\in\N$ and this conclude the proof of Theorem \ref{th: stabilization}.


\begin{thebibliography}{99}

\bibitem{Ala-Can-Com} {\sc F.~Alabau, P.~Cannarsa, and V.~Komornik,} Indirect internal stabilization of weakly coupled systems, {\em J. Evolution Equations,} {\bf 2} (2002), 127--150.

\bibitem{Ala-Can-Leu}  {\sc F.~Alabau-Boussouira, P.~Cannarsa and G.~Leugering}, Control and stabilization of degenerate wave equations, {\em SIAM J. Control Optim.,} {\bf 55} (2017), 2052--2087.

\bibitem{AHR} {\sc K.~Ammari, F.~Hassine and L.~Robbiano,} Stabilization for the wave equation with singular Kelvin-Voigt damping, {\em Arch. Ration. Mech. Anal.,} {\bf 236} (2020), 577--601.

\bibitem{AH} {\sc K.~Ammari and F.~Hassine,} {\em Stabilization of Kelvin-Voigt damped systems}, Adv. Mech. Math., 47, Birkh\"auser Springer, Cham, 2022.

\bibitem{ALS} {\sc K.~Ammari, Z.~Liu and F.~Shel,} Stability of the wave equations on a tree with local Kelvin-Voigt damping, {\em Semigroup Forum}, {\bf 100} (2020), 364--382.

\bibitem{Amm-Liu-Shel} {\sc K.~Ammari, Z.~Liu and F.~Shel,} Note on stability of an abstract coupled hyperbolic-parabolic system: Singular case, {\em Applied Mathematics Letters} {\bf 141} (2023), 108599.

\bibitem{Amm-Bad-Ben} {\sc F.~Ammar-Khodja, A.~Bader, and A.~Benabdallah,} Dynamic stabilization of systems via decoupling techniques, {\em ESAIM Control Optim. Calc. Var.,} {\bf 4} (1999), 577--593.

\bibitem{Are-Bat} {\sc W.~Arendt and C.J.K.~Batty,} Tauberian theorems and stability of one-parameter semigroups, {\em Trans. Amer. Math. Soc.,} {\bf 306} (1988), 837--852.

\bibitem{Ata-Kam} {\sc A.~Atallah-Baraket and  C.F.~Kammerer,} Analysis of the energy decay of a degenerated thermoelasticity system, {\em arXiv preprint arXiv:1203.5606,} (2012).

\bibitem{Ava-Las-Tri} {\sc G.~Avalos, I.~Lasiecka and R.~Triggiani,} Heat-wave interaction in 2-3 dimensions: Optimal rational decay rate, {\em J. Math. Anal. Appl.,} {\bf 437} (2016), 782--815.

\bibitem{Bat-Pau-Sei} {\sc C.J.K.~Batty, L.~Paunonen, and D.~Seifert,} Optimal energy decay in a one-dimensional coupled wave-heat system, {\em J. Evol. Equ.,} {\bf 16} (2016), 649--664.

\bibitem{Bey} {\sc A. Beyrath,} Stabilisation indirecte intense par un feedback localement distribu\'e de systemes d'equations coupl\'es, {\em C. R. Acad. Sci. Paris S\'er. I Math.,} {\bf 333} (2001), 451--456.

\bibitem{Bor} {\sc R.D.~Borcherdt,} {\em Viscoelastic Waves in Layered Media,} Cambridge University Press, Cambridge, 2009.


\bibitem{Chain-Trig} {\sc S.P.~Chen and R.~Triggiani} Proof of extensions of two conjectures on structural damping for
elastic systems, {\em Pacific J. Math.,} {\bf 136} (1989), 15--55.

\bibitem{CVC1} {\sc S.~Carillo, V.~Valente and G.V.~Caffarelli,} A result of existence and uniqueness for an integro-differential system in magneto-viscoelasticity, {\em Apll. Ana.,} {\bf 90}, (2010), 1791--1802.

\bibitem{CVC2} {\sc S.~Carillo, V.~Valente and G.V.~Caffarelli,} A linear viscoelasticity problem with a singular memory kernel: an existence and uniqueness result, {\em Different. Integr. Equat.,} {\bf 26} (2013), 1115--1125.

\bibitem{C} {\sc S.~Carillo,} A $3$-dimensional singular kernel problem in viscoelasticity: an existence result, {\em Atti della Accademia Peloritana dei Pericolanti, Classe di scienze Fisiche. Matematiche Naturali,} {\bf 97} (2019), 1--13.

\bibitem{Car-Fra-Roc} {\sc P.~Cannarsa, G.~Fragnelli and  D.~Rocchetti,} Controllability results for a class of one-dimensional degenerate parabolic problems in nondivergence form, {\em J. Evol. Equ.,} {\bf 8} (2008), 583--616.

\bibitem{Car-Mar-Van3} {\sc P.~Cannarsa, P.~Martinez and J.~Vancostenoble,} Carleman estimates for a class of degenerate parabolic operators, {\em SIAM J. Control Optim.,} {\bf 47} (2008), 1--19.

\bibitem{Car-Mar-Van1} {\sc P.~Cannarsa, P.~Martinez and J.~Vancostenoble,} Null controllability of degenerate heat equations, {\em Adv. Differ. Equ.,} {\bf 10}  (2005), 153--190.

\bibitem{Car-Mar-Van2} {\sc P.~Cannarsa, P.~Martinez and J.~Vancostenoble,} Persistent regional null controllability for a class of degenerate parabolic equations, {\em Commun. Pure Appl. Anal.,} {\bf 3} (2004) , 607--635,

\bibitem{Chen-Liu} {\sc W.~Chen and Y.~Liu,} Large time behavior for the hyperbolic-parabolic coupled system with the regularity-loss structure, {\em  arXiv:2011.07757v2 [math.AP]} (2024).

\bibitem{Daf} {\sc C.M.~Dafermos,} On the existence and the asymptotic stability of solutions to the equations of linear thermoelasticity, {\em Arch. Ration. Mech. Anal.,} {\bf 29} (1968), 241--271.

\bibitem{Duc} {\sc T. Duyckaerts,} Optimal decay rates of the energy of a hyperbolic-parabolic system coupled by an interface, {\em Asymptot. Anal.,} {\bf 51} (2007), 17--45.

\bibitem{Fer-Liu-Racke} {\sc H.D.~Fern\'andez, Z.~Liu and R.~Racke,} Stability of abstract thermoelastic systems with inertial terms, {\em J. Differential Equations,} {\bf 267} (2019), 7085--7134.

\bibitem{Gao-Li-Liu} {\sc H.~Gao, L.~Li and Z.~Liu} Stability of degenerate heat equation in non-cylindrical/cylindrical domain, {Z. Angew. Math. Phys.,} {\bf 70} (2019), 1--17.

\bibitem{Gueye}  {\sc M.~Gueye}, Exact boundary controllability of 1-D parabolic and hyperbolic degenerate equations, {\em SIAM J. Control Optim.,} {\bf 52} (2014), 2037--2054.

\bibitem{Han-Wang-Wang} {\sc Z.J.~Han, G.~Wang, and J.~Wang,} Explicit decay rate for a degenerate hyperbolic-parabolic
coupled system, {\em ESAIM Control Optim. Calc. Var.,} {\bf 26} (2020), 116.

\bibitem{Hao-Liu} {\sc J.~Hao and Z.~Liu,} Stability of an abstract system of coupled hyperbolic and parabolic equations, {\em Z. Angew. Math. Phys.,} {\bf 64} (2013), 1145--1159.

\bibitem{Hao-Liu-Yong} {\sc J.~Hao, Z.~Liu and J.~Yong,} Regularity analysis for an abstract system of coupled hyperbolic and parabolic equations, {\em J. Differential Equations,} {\bf 259} (2015), 4763--4798.

\bibitem{hassine1}{\sc F.~Hassine,}  Stability of elastic transmission systems with a local Kelvin–Voigt damping, {\em European Journal of Control,} {\bf 23} (2015), 84--93.

\bibitem{hassine2}{\sc F.~Hassine,}  Asymptotic behavior of the transmission Euler-Bernoulli plate and wave equation with a localized Kelvin-Voigt damping, {\em Discrete and Continuous Dynamical Systems - Series B,} {\bf 21} (2016), 1757--1774.

\bibitem{hassine4}{\sc F.~Hassine,}  Energy decay estimates of elastic transmission wave/beam systems with a local Kelvin-Voigt damping, {\em Internat. J. Control,} {\bf 89} (10) (2016), 1933--1950.

\bibitem{hassine3}{\sc F.~Hassine,}  Logarithmic stabilization of the Euler-Bernoulli plate equation with locally distributed Kelvin-Voigt damping, {\em Evolution Equations and Control Theory,} {\bf 455}(2) (2017), 1765--1782.

\bibitem{huang} {\sc F.~Huang,} Characteristic conditions for exponential stability of linear dynamical systems in Hilbert space, {\em Ann. Differential Equations.,} {\bf 1} (1985), 43--56.

\bibitem{Kap} {\sc B.V.~Kapitonov,} Uniform stabilization and exact controllability for a class of coupled hyperbolic systems, {\em Comput. Appl. Math.,} {\bf 15} (1996), 199--212.

\bibitem{Kin-Par} {\sc J.~Kinnunen and M.~Parviainen,} Stability for degenerate parabolic equations, {\em Adv. Calc. Var.,} {\bf 3} (2010), 29--48.

\bibitem{Las} {\sc I. Lasiecka,} Uniform decay rates for full von Karman system of dynamic thermoelasticity with free boundary conditions and partial boundary dissipation, {\em Comm. Partial Differential Equations,} {\bf 24} (1999), 1801--1847.

\bibitem{Leb-Zua} {\sc G.~Lebeau and E.~Zuazua,} Decay rates for the three-dimensional linear system of thermoelasticity, {\em Arch. Rational Mech. Anal.,} {\bf 148} (1999), 179--231.

\bibitem{Liu-Liu} {\sc K. Liu and Z.~Liu}, Exponential stability and analyticity of abstract linear thermoelastic systems, {\em Z. Angew. Math. Phys.,} {\bf 48} (1997), 885--904.

\bibitem{MKPK} {\sc M.~Mahiuddin, M.I.H.~Khan, N.D.~Pham and M.A.~Karim} Development of fractional viscoelastic model for characterizing viscoelastic properties of food material during drying, {\em Food bioscience,} {\bf 23} (2018), 45--53.

\bibitem{Muniz-Racke} {\sc J.E.~Mu\~noz Rivera and R.~Racke,} Smoothing properties, decay, and global existence of solutions to nonlinear coupled systems of thermoelastic type, {\em SIAM J. Math. Anal.,} {\bf 26} (1995), 1547--1563.

\bibitem{Ng-Sei} {\sc A.C.S.~Ng and D.~Seifert,} Optimal energy decay in a one-dimensional wave-heat system with infinite heat part, {\em J. Math. Anal. Appl.,} {\bf 482} (2020), 123563.

\bibitem{Now} {\sc W.~Nowacki,} {\em Thermoelasticity,} 2nd ed., PWN-Polish Sci. Publ., Warszawa 1986.

\bibitem{Ole-Sam} {\sc O.A.~Oleinik and V.N.~Samokhin,} {\em Mathematical Models in Boundary Layer Theory}, Applied Mathematics and Mathematical Computation 15. Chapman and Hall/CRC, Boca Raton 1999.

\bibitem{Pazy} {\sc A.~Pazy,} {\em Semigroups of linear operators and applications to partial differential equations}, Springer, New York, 1983.

\bibitem{pruss} {\sc J.~Pr\"uss,} On the spectrum of $C_0$-semigroups, {\em Trans. Amer. Math. Soc.,} {\bf 284} (1984), 847--857.

\bibitem{RFG} {\sc A.~Rassoli, N.~Fatouraee and R.~Guidoin,} Structural model for viscoelastic properties of pereicardial bioprosthitic valves, {\em Artif. Organs.,} {\bf 42} (2018), 630--639.

\bibitem{SA} {\sc A.~Sair, and A.~Ambrosetti,} {\em A textbook on ordinary differential equations}, (Vol. 88), Springer, 2015.

\bibitem{SHTBSA} {\sc A.~Shahin-Shamsabadi, A.~Hashemi, M.~Tahriri, F.~Bastami, M.~Salehi and F.M.~Abbas,} Mechanical, material, and biological study of a PCL/bioactive glass bone scaffold: Importance of viscoelasticity, {\em Materials Science and Engineering: C,} {\bf 90} (2018), 280--288.


\bibitem{TW} {\sc M.~Tucsnak and G.~Weiss,} {\em Observation and control for operator semigroups}, Birkh\"auser Verlag AG, 2009.

\bibitem{Teb1} {\sc L.~Tebou,} Sharp decay estimates for semigroups associated with some one-dimensional fluid-structure interactions involving degeneracy, {\em SIAM J. Control Optim.,} {\bf 60} (2022), 2787--2810.

\bibitem{Teb2} {\sc L.~Tebou,} Stability and Gevrey regularity for some transmission systems involving a degenerate parabolic component, {\em Journal of Evolution Equations} {\bf 23} (2023), 26.

\bibitem{Wang1} {\sc C.~Wang,} Boundary behavior and asymptotic behavior of solutions to a class of parabolic equations with boundary degeneracy, {\em Discrete Contin. Dyn. Syst.} {\bf 36} (2016), 1041--1060.

\bibitem{Wang2} {\sc C.~Wang,} Approximate controllability of a class of degenerate systems, {\em Appl. Math. Comput.,} {\bf 203} (2008), 447--456,

\bibitem{Zha-Zua} {\sc X.~Zhang and E.~Zuazua,} Long-time behavior of a coupled heat-wave system arising in fluid structure interaction, {\em Arch. Ration. Mech. Anal.,} {\bf 184} (2007), 49--120.
\end{thebibliography}
\end{document}